\documentclass[10pt,a4paper]{article}

\usepackage[british]{babel}
\usepackage[utf8]{inputenc}

\usepackage{tabularx}

\usepackage{amssymb,amsmath,amsthm,thmtools}
	\numberwithin{equation}{section}
\usepackage{mathtools}	
\usepackage{accents}
\usepackage{multirow}
\usepackage{graphicx}
\usepackage{amscd}
\usepackage{epic, eepic}
\usepackage{url}
\usepackage{xcolor}
\usepackage{todonotes}
\usepackage{enumerate}
\usepackage{enumitem}
\usepackage{comment}
\usepackage{hhline}
\usepackage[hidelinks]{hyperref}
    \addto\extrasbritish{%
    }

\usepackage{tikz-cd}

\counterwithin{figure}{section}

\def\F{\mathbb F}

\newcommand{\cC}{\mathcal{C}}

\newcommand{\cS}{\mathcal{S}}
\newcommand{\cT}{\mathcal{T}}

\newcommand{\PG}{\mathrm{PG}}

 \makeatletter
 \@addtoreset{equation}{section}
 \makeatother

 \makeatletter
 \@addtoreset{table}{section}
 \makeatother
 
\declaretheorem[style=plain,name=Theorem,numberwithin=section]{theorem}
\declaretheorem[style=plain,name=Proposition,sibling=theorem]{prop}
\declaretheorem[style=plain,name=Lemma,sibling=theorem]{lemma}
\declaretheorem[style=plain,name=Problem,sibling=theorem]{problem}
\declaretheorem[style=plain,name=Corollary,sibling=theorem]{cor}

\declaretheorem[style=definition,name=Definition,sibling=theorem]{defn}

\declaretheorem[style=definition,name=Remark,sibling=theorem]{remark}

\newcommand{\Projectivisation}[1]{\mathbb{P}(#1)}
\newcommand{\V}[1]{\mathbb{V}(#1)}

\newcommand{\cardinality}[1]{\left|#1\right|}

\newcommand{\GrassmannDimSpace}[2]{\mathcal{G}_{#1}(#2)}
\newcommand{\GrassmannSpace}[1]{\GrassmannDimSpace{\bullet}{#1}}

\newcommand{\grassmannDimSpace}[2]{\cardinality{\GrassmannDimSpace{#1}{#2}}}

\newcommand{\qbinom}[3]{\genfrac{[}{]}{0pt}{}{#1}{#2}_{\!#3}}

\newcommand{\field}[1]{{\mathbb{F}_{#1}}}

\newcommand{\leteq}{\coloneqq}

\newcommand{\generated}[1]{\langle #1 \rangle}

\newcommand{\isom}{\cong}
\DeclareMathOperator{\im}{im}

\DeclareMathOperator{\id}{id}

\newcommand{\bigO}[1]{O(#1)}

\newcommand{\setBuilder}[2]{\left\{#1 : #2 \right\}}%

\newcommand{\TuranQNTS}[4]{\mathcal{T}_{#1}(#2, #3, #4)}

\newcommand{\blocking}{\mathcal{B}}
\newcommand{\blockingSets}[3]{\mathcal{F}(#3)}

\newcommand{\density}{\varrho}

\newcommand{\from}{\colon}

    \makeatletter
    \providecommand*{\cupdot}{%
      \mathbin{%
        \mathpalette\@cupdot{}%
      }%
    }
    \newcommand*{\@cupdot}[2]{%
      \ooalign{%
        $\m@th#1\cup$\cr
        \sbox0{$#1\cup$}%
        \dimen@=\ht0 %
        \sbox0{$\m@th#1\cdot$}%
        \advance\dimen@ by -\ht0 %
        \dimen@=.5\dimen@
        \hidewidth\raise\dimen@\box0\hidewidth
      }%
    }
    
    \providecommand*{\bigcupdot}{%
      \mathop{%
        \vphantom{\bigcup}%
        \mathpalette\@bigcupdot{}%
      }%
    }
    \newcommand*{\@bigcupdot}[2]{%
      \ooalign{%
        $\m@th#1\bigcup$\cr
        \sbox0{$#1\bigcup$}%
        \dimen@=\ht0 %
        \advance\dimen@ by -\dp0 %
        \sbox0{\scalebox{2}{$\m@th#1\cdot$}}%
        \advance\dimen@ by -\ht0 %
        \dimen@=.5\dimen@
        \hidewidth\raise\dimen@\box0\hidewidth
      }%
    }
    \makeatother

\newcommand{\disjointunion}{\cupdot}
\newcommand{\disjointUnion}{\bigcupdot}

\newcommand{\opt}{\mathrm{opt}}
\newcommand{\lex}{\mathrm{lex}}

\newcommand{\embeds}{\hookrightarrow}
\newcommand{\surjects}{\twoheadrightarrow}
\newcommand{\polarity}[1]{\phi_{#1}}

\newcommand{\coefficients}[4]{%
    {#1}\cdot q^{2n-4} + %
    {#2}\cdot  q^{2n-5} + %
    {#3}\cdot q^{2n-6} + %
    {#4}\cdot q^{2n-7} + %
    \bigO{q^{2n-8}}}

\DeclareMathOperator{\Dim}{Dim}

\title{Blocking Planes by Lines in $\PG(n,q)$}
\author{Benedek Kovács\thanks{ELTE Linear Hypergraphs Research Group, Eötvös Loránd University, Budapest, Hungary. Supported by the EKÖP-24 University Excellence Scholarship Program of the Ministry for Culture and Innovation from the source of the National Research, Development and Innovation Fund and the University Excellence Fund of Eötvös Loránd University.
		E-mail: {\tt benoke98@student.elte.hu}}
	\and Zolt\'an L\'or\'ant Nagy\thanks{ELTE Linear Hypergraphs  Research Group,
		E\"otv\"os Lor\'and University, Budapest, Hungary. The author is supported by the Hungarian Research Grant (NKFIH) No. PD  134953. and No. K. 124950 and the University Excellence Fund of Eötvös Loránd University.	E-mail: {\tt nagyzoli@cs.elte.hu}}
  \and Dávid R. Szabó\thanks{ELTE Linear Hypergraphs  Research Group,
		E\"otv\"os Lor\'and University, Budapest, Hungary. The author is supported by the University Excellence Fund of Eötvös Loránd University. E-mail: {\tt szabo.r.david@gmail.com}}}

\begin{document}

\maketitle

\begin{abstract}
In this paper, we study the cardinality of the smallest set of lines of the finite projective spaces $\PG(n, q)$ such that every plane is incident with at least one line of the set. This is the first main open problem concerning the minimum size of $(s,t)$-blocking sets in $\PG(n,q)$, where we set $s=2$ and $t=1$. 
 In $\PG(n,q)$, an $(s,t)$-blocking set refers to a set of $t$-spaces such that each $s$-space is incident with at least one chosen $t$-space. This is a notoriously difficult problem, as it is equivalent to determining the size of certain $q$-Turán designs and $q$-covering designs.
 We present an improvement on the upper bounds of Etzion and of Metsch via a refined scheme for a recursive construction, which in fact enables improvement in the general case as well.
\end{abstract}

\section{Introduction}

Let $q$ be a prime power and let $\F_q$ denote the finite field of $q$ elements.
 $\PG(n, q)$ denotes the $n$-dimensional projective space over $\F_q$, while a subspace of dimension $d$ will be called a $d$-space.
 
 In the history of finite geometry, one of the most  central problems is to describe the size and structure of {\em blocking sets}. A blocking set with respect to $s$-dimensional subspaces is a
 point set of  $\PG(n, q)$ which meets every $s$-dimensional subspace in at least one point.
 As observed by Bose and Burton,  the point set of an $(n-s)$-dimensional subspace is the smallest possible
 blocking set with respect to $s$-dimensional subspaces \cite{BB}. From this classical result, a whole theory of so-called Bose-Burton type theorems has emerged.
 Indeed, the above concept has  some far-reaching generalisations, when one can take sets of subspaces of a given dimension $t$ instead of point sets (which correspond to the case $t=0$).
  Let $t \le s$ be nonnegative integers.
   In this general context, an {\em $(s, t)$-blocking set} $B$ is a set of $t$-spaces of $\PG(n,q)$ such that every $s$-space of $\PG(n,q)$ contains at least one element of $B$. 

An example for $(s, t)$-blocking sets can be obtained from 
geometric $t$-spreads, first introduced by André \cite{Segre}.

\begin{defn}[$t$-spread]  A \emph{$t$-spread} of a projective space $\PG(n,q)$
 is a set of $t$-dimensional subspaces for some $t<n$, such that every point of the space lies in exactly one of the elements of the spread.
\end{defn}
The necessary and sufficient condition for the existence of $t$-spreads in $\PG(n,q)$ is that $(t+1) \mid (n+1)$.

Beutelspacher and Ueberberg found the minimal $(s,t)$-blocking sets in the following cases.

\begin{theorem}[Beutelspacher and Ueberberg, \cite{beutel}]\label{thm:beutel}
    Suppose that $\blocking$ is an $(s,t)$-blocking set. If
$n \le s +s/t-1$ or $t=0$, then 
$$|\blocking|\ge (q^{(t+1)(n+1-s)}-1)/(q^{t+1}-1),$$
where  equality is attained if and only if $\blocking$ is a so-called geometric $t$-spread in a subspace of dimension equal to $(n + 1 -s)(t + 1) -1$.
\end{theorem}

In the last fifty years the general problem of determining the smallest cardinality of a blocking set has been studied by several authors both in the original case when $t=0$ and the general case ($t\in \mathbb{N}$) (see \cite{BSZ-survey, Metsch2} and references therein). A connection to maximum rank distance codes \cite{Etz2} and to $q$-designs \cite{Braun, Etz1} has been investigated \cite{Etz2}, and these topics have been extensively studied later on (see \cite{Gabi}, \cite{Csajbok}, and \cite{Ozbudak}, and the references therein).
 In recent years, additional motivation for such results has been highlighted by the
increasing interest in codes over the Grassmannian as a result of their application to error correction in network coding as was demonstrated by Koetter and Kschischang \cite{KK}.
Even so, only a few instances are known where the problem is  completely solved \cite{BB,  Metsch, Pavese}.

\paragraph{Known results in the case $t=1$.}

Consider the problem of finding minimal $(s,1)$-blocking sets in $\PG(n,q)$. 
\autoref{thm:beutel} solves the problem on the interval $s\leq n\leq 2s-1$. 
In the case of the subsequent interval $2s\leq n\leq 3s-3$, by pulling back the $(q^{2s}-1)/(q^2-1)$ lines forming a minimal $(s,1)$-blocking set in a suitable $2s-1$-dimensional quotient space, Metsch constructed an $(s,1)$-blocking set in $\PG(n,q)$.  
This construction is minimal and essentially is the unique minimal one. 
\begin{theorem}[Eisfeld and Metsch, {\cite[Theorem 1.2.]{eis}, \cite[Theorem 1.2]{Metsch}}]\label{thm:n<3s-2}
    Let $2s\leq n\leq 3s-3$ for $s\geq 3$, or $n=4$ for $s=2$. 
    Let $q$ be a prime power. 
    Write $k\leteq n-2s$. 
    Assume further that $q\geq 4^{k+1}+2^{k+1}+1$ if $n\neq 2s$. 
    Then every $(s,1)$-blocking set $\blocking$ in $\PG(n,q)$ satisfies \[\cardinality{\blocking}\geq q^{2k+2}\cdot \frac{q^{2s}-1}{q^2-1} + \sum_{i=0}^{k}q^i\sum_{j=k}^{n-s}q^j.\]
    Furthermore, there are blocking sets $\blocking$ with equality above, and the structure of such sets is known. 
\end{theorem}

\paragraph{The case $(s,t)=(2,1)$.}

In the examples above, $n$ is bounded. 
Our main intention is to determine the size of $(s,t)$-blocking sets in projective spaces of higher dimension, at least when $(s,t)=(2,1)$. 
Even in this case, the problem is solved only in low dimensions.
Our main function will be as follows.
\begin{defn}\label{defn:f}
Let $f(n,q)$ denote the smallest possible size of a $(2,1)$-blocking set in $\PG(n,q)$ for $n\geq -1$.
\end{defn}
\begin{remark}
    Note that we have $f(-1,q)=f(0,q)=f(1,q)=0$, since in these degenerate cases,  there are no $2$-dimensional subspaces in $\PG(n,q)$.
\end{remark}
 
Apart from the trivial case $n=2$, \autoref{thm:beutel} and \autoref{thm:n<3s-2} gives the following known values.
\begin{theorem}\label{thm:f234}
For every prime power $q$, we have 
$f(2,q)=1$,
$f(3,q)=q^2+1$, 
$f(4,q)=q^4+2q^2+q+1$.
\end{theorem}

While the proof for $n=3$ was not very involved, the $4$-dimensional case required a tour-de-force proof \cite{eis}.
In fact, Metsch \cite{Metsch} notes that his refined double counting technique even leads to an exact result for $n=5$ as well with $f(5,q)=q^6+2q^4+2q^3+2q^2+q$ for large enough $q$, but the argument remained unpublished.

Let us point out that these results translate to results concerning $q$-covering designs and $q$-Turán designs, which will be detailed in the next section.
Note that $f(n,q)$ is the minimum size of a $q$-Turán design $T_q[n+1,3,2]$.

In general, the $q$-analog Schönheim bound, formulated by Etzion and Vardy  \cite{Etz2}, gives the best known lower bound. It provides a lower bound for every $(s,t)$-blocking set (see also in Section \ref{sect:lowerb}); below we apply it to the case $(s,t)=(2,1)$.  

\begin{theorem}[Etzion, Vardy, \cite{Etz2}]\label{also}
$$f(n,q)\geq \left\lceil \frac{q^{n+1}-1}{q^{n-1}-1}\cdot f(n-1,q)\right\rceil.$$
\end{theorem}

If one applies  \autoref{also} recursively and omits the ceiling, this leads to

\begin{prop}[Lower bound]\label{main1}
For every $n\geq 4$ and prime power $q$, we have
 \[f(n,q)\geq (q^{n+1}-1)(q^n-1)\cdot \frac{q^4+2 q^2+q+1}{(q^5-1)(q^4-1)}.\]
In particular, estimating the rational function from below for $n\geq 5$ gives  \[f(n,q)\ge q^{2n-4}+2q^{2n-6}+q^{2n-7}+2q^{2n-8}.\]
\end{prop}

Moreover, building on the claim of Metsch \cite{Metsch} for $f(5,q)$, we would also have $$f(n,q)\ge q^{2n-4}+2q^{2n-6}+2q^{2n-7}+2q^{2n-8}.$$

Our contribution is an improvement on the upper bounds according to the result below, which roughly matches the bound above.

\begin{theorem}[Upper bound]\label{main2}
If $n\ge 2$ is fixed and $q$ is a variable prime power, then $$f(n,q)\le q^{2n-4}+2q^{2n-6}+2q^{2n-7}+3q^{2n-8}+    {3} q^{2n-9}+
    {3} q^{2n-10}+
 {3} q^{2n-11}+
    \bigO{q^{2n-12}}.$$
\end{theorem}
Observe that this upper bound is tight up to the first four leading terms.
Our recursive construction for $f(n,q)$ (see \autoref{recursion}) uses a parameter $k$. The construction in the cases $k\in \{0, 1\}$ specialises to that of Eisfeld's and Metsch's \cite{eis,Metsch}, which obtains bounds on $f(n,q)$ in dimensions $n=4$ and $n=5$ relying on the case of $n=3$ when we know the exact result via Theorem \ref{thm:beutel}. We show that one can improve upon this if we start off the recursion from dimension $n=3$, and crucially, instead of augmenting the dimension by one, we use our recursion with a dimension step depending on $n$, which enables us to block the planes more efficiently.

The paper is organised as follows.
In \autoref{sect:prelim} we set our main technical notations and the relation to the general question on $(s,t)$-blocking sets to $q$-Turán and subspace designs.  In \autoref{sect:constr} we discuss the recursive construction, which has origins from the papers of Metsch \cite{Metsch} and of Etzion and Vardy \cite{Etz2}.
In \autoref{sec:3.1} we propose a new recursive construction in \autoref{cor:explicitRecusrionST}  for general $(s,t)$-blocking sets. In \autoref{sec:3.2} we specialize to the case $(s,t)=(2,1)$ and improve our previous general construction further.
We point out how this constructive approach generalizes the one applied by Eisfeld and Metsch \cite{eis}, and then, further using the ideas of Metsch \cite{Metsch}, this  enables further improvement, presented in \autoref{sec:construction21}.  In \autoref{sec:recursiveUpperBOund}, we show how to use the recursion optimally in order to achieve the best upper bound via this recursive scheme. Next, in \autoref{sect:lowerb} we discuss the lower bounds and their possible improvement. Finally, the last section is dedicated to  some related problems and open questions on $q$-designs.
\section{Preliminaries}\label{sect:prelim}

\subsection{Notation, related results}

First, we introduce the concepts of $q$-covering designs, $q$-Turán designs, $q$-Steiner designs and subspace designs in general and their relations to each other.
These can be investigated as structures on the linear space, but also as structures on the projective space. Observe that there is a difference of $1$
between the dimension of a subspace of the linear space and the dimension of the corresponding subspace in the projective geometry. Indeed, (projective) $(k-1)$-subspaces of $\PG(n-1, q)$ correspond to $k$-subspaces of $\F_q^{n}$.

\begin{defn}    Let $r\le k\le n$.
\begin{itemize}
\item  A \textbf{$q$-covering design $C_q[n, k, r]$} is a set  $S$ of $k$-spaces of an $n$-dimensional vector space, such that each
 $r$-space of the vector space \textbf{is contained in at least one element of $S$}.  Let $\cC_q(n,k,r)$ denote the minimum number of $k$-spaces in a $q$-covering design ${C}_q(n,k,r)$.
\item A \textbf{$q$-Turán design $T_q[n, k, r]$} is a collection $S$ of $r$-spaces of an $n$-dimensional vector space, such that each
 $k$-space of the vector space  \textbf{contains at least one element of $S$}.  Let $\cT_q(n,k,r)$ denote the minimum number of $r$-spaces in a $q$-Turán design $T_q(n,k,r)$.
\item A \textbf{$q$-Steiner design $S_q[n, k, r]$} is a collection $S$ of $k$-spaces of an $n$-dimensional vector space, such that each
 $r$-space of the vector space \textbf{is contained in exactly one element of $S$}.  
\item Finally, an  \textbf{$(n, k, r, \lambda)_q$ subspace design} is a collection $S$ of $k$-spaces of an $n$-dimensional vector space, such that each
 $r$-space of the vector space  \textbf{is contained in exactly $\lambda$ elements of $S$}. 
\end{itemize}
\end{defn}
We recall the following well-known observation.
\begin{prop}[Duality between coverings and Turán designs]
An $(n, k, r, \lambda)_q$ subspace design is a $q$-Steiner design if $\lambda=1$. Since the dual structure of a $q$–Turán design $T_q(n, k, r)$ is a q–covering design  $C_q(n, n-r, n-k)$, we have $$\cT_q(n, k, r)=\cC_q(n, n-r, n-k).$$ Hence, a $q$-Steiner design is a $q$-covering design which corresponds also to a $q$-Turán design with appropriate parameters.
\end{prop}

Using these notations, in order to study minimal-size blocking sets $ B$  of lines of $\PG(n-1,q)$ such that every plane of $\PG(n-1,q)$ contains at least one element of $B$, we wish to improve bounds on $\cT_q(n, 3, 2)=\cC_q(n, n-2, n-3)$.

To help the reader to distinguish different dimension notions, from now on, $\Dim(V)$ will denote the \emph{vector space dimension} when $V$ is introduced as a vector space, while $\dim(X)$ will refer to the \emph{projective dimension} for a projective space $X$.

\begin{defn}[Grassmannian]\label{def:grassmann}
Let  $X$ be a projective space over a field $\field{}$.
For $-1\leq d\leq \dim(X)$, let  
\[\GrassmannDimSpace{d}{X}\leteq \setBuilder{Y\subseteq X}{\text{$Y$ is a $d$-dimensional projective space (over $\field{}$)}},\]
the \emph{Grassmannian} of projective subspaces of dimension $d$. 
Write $\GrassmannSpace{X}=\bigcup_{d=-1}^{\dim(X)}\GrassmannDimSpace{d}{X}$ for the lattice of subspaces.
Note that $\GrassmannDimSpace{-1}{X}$ is the singleton set containing the unique smallest element of the lattice $\GrassmannSpace{X}$, e.g. $\dim(\emptyset)$ is considered to be $-1$.
Write $\GrassmannDimSpace{d}{n,q}$ as a shorthand for $\GrassmannDimSpace{d}{\PG(n,q)}$. 
See \autoref{fig:maps}.
\end{defn}

\begin{defn}\label{defn:Proj}
    Let $V$ be a vector space over a field $\field{}$. Its \emph{projectivisation} $\Projectivisation{V}$ is defined by \[\GrassmannDimSpace{d}{\Projectivisation{V}}\leteq \setBuilder{U\leq V}{\Dim_V(U)=d+1}\] for $-1\leq d\leq \dim(\Projectivisation{V})\leteq \Dim(V)-1$.
    We write $\V{K}$ for $K$ to indicate if we view it as a vector subspace of $V$, instead of an element of $\GrassmannSpace{\Projectivisation{V}}$. Call $\V{K}\leq V$, the \emph{subspace corresponding to $K\in \GrassmannSpace{\Projectivisation{V}}$}.
\end{defn}

\begin{remark}\label{rem:ProjAndVinverses}
   The projectivisation  $X\leteq\Projectivisation{V}$ of $V$ is a projective space over $\field{}$. For a subspace $K\in\GrassmannSpace{X}$,  
    $\V{K}$ is a subspace of $V$ of dimension $\Dim_V(\V{K}) = \dim_X(K)+1$.  
    These two operations are inverses of each other and give a bijection between the subspace lattice of $V$ and $\GrassmannSpace{\Projectivisation{V}}$.
\end{remark}

Next, we define the Gaussian binomial coefficients which count the number of subspaces of given dimension in a vector space, also allowing us to compute the size of the Grassmannians in projective spaces.

\begin{defn} For non-negative integers $m,d$ and a prime power $q$, the Gaussian binomial coefficients are defined by
$$\qbinom{m}{d}{q}
\leteq \begin{cases}
\frac{(q^m-1)(q^{m-1}-1)\cdots(q^{m-d+1}-1)} {(q^d-1)(q^{d-1}-1)\cdots(q-1)} & d \le m \\
0 & d>m. \end{cases}$$
\end{defn}

\begin{defn}[Quotient space]\label{def:quotientGeometry}
   Let $X$ be a projective space. For $K\in\GrassmannDimSpace{k}{X}$, the \emph{quotient space $X/K$} is a projective space of dimension $\dim(X/K)=\dim(X)-k-1$ given by 
    \[\GrassmannDimSpace{r}{X/K} \leteq \{Y\in\GrassmannDimSpace{r+k+1}{X}:K\subseteq Y\}\] for any $-1\leq r\leq \dim(X/K)$. 
    Note that $\GrassmannDimSpace{-1}{X/K}=\{K\}$. 
    See \autoref{fig:maps}.
\end{defn}

Quotients and projectivisations are related as follows.
\begin{remark}\label{rem:ProjQuotientCommute}
    For any vector subspace $U\leq V$, and any $-1\leq r\leq \dim(\Projectivisation{V}/\Projectivisation{U})$, we have 
    \begin{align*}
        \GrassmannDimSpace{r}{\Projectivisation{V}/\Projectivisation{U}} & = \setBuilder{Y\in\GrassmannDimSpace{r+\dim(U)}{\Projectivisation{V}}}{\Projectivisation{U}\subseteq Y} =
        \setBuilder{W\leq V}{\dim_V(W) = r+\dim(U)+1,U\leq W} 
        \\&\isom \GrassmannDimSpace{r}{\Projectivisation{V/U}}
    \end{align*}
    upon identifying $r$-dimensional subspaces of $V/U$ with $(r+\dim(U))$-dimensional subspaces of $V$ containing $U$. 
    In particular, $\Projectivisation{V}/\Projectivisation{U}$ can be identified with $\Projectivisation{V/U}$.
\end{remark}

\begin{remark}\label{rem:GrassmannianCardinality}  Let $X$ be a projective space over $\field{q}$. 
Note that $\grassmannDimSpace{d}{X} = \qbinom{\dim(X)+1}{d+1}{q}$, or more generally,  
$\grassmannDimSpace{d}{X/K} = \qbinom{\dim(X)-\dim(K)}{d+1}{q}$. 
In particular, the number of $d$-dimensional subspaces of $X$ containing a fixed $K\in\GrassmannDimSpace{k}{X}$ is 
$\grassmannDimSpace{d-k-1}{X/K} = 
\qbinom{\dim(X)-\dim(K)}{d-\dim(K)}{q}$.
\end{remark}

Finally, we use the notation $[a,b]$ for the set of integers $\{k\in \mathbb{Z}: a\le k\le b\}$.

\begin{figure}[!htb]
    \centering
    \includegraphics[height=0.5\textheight]{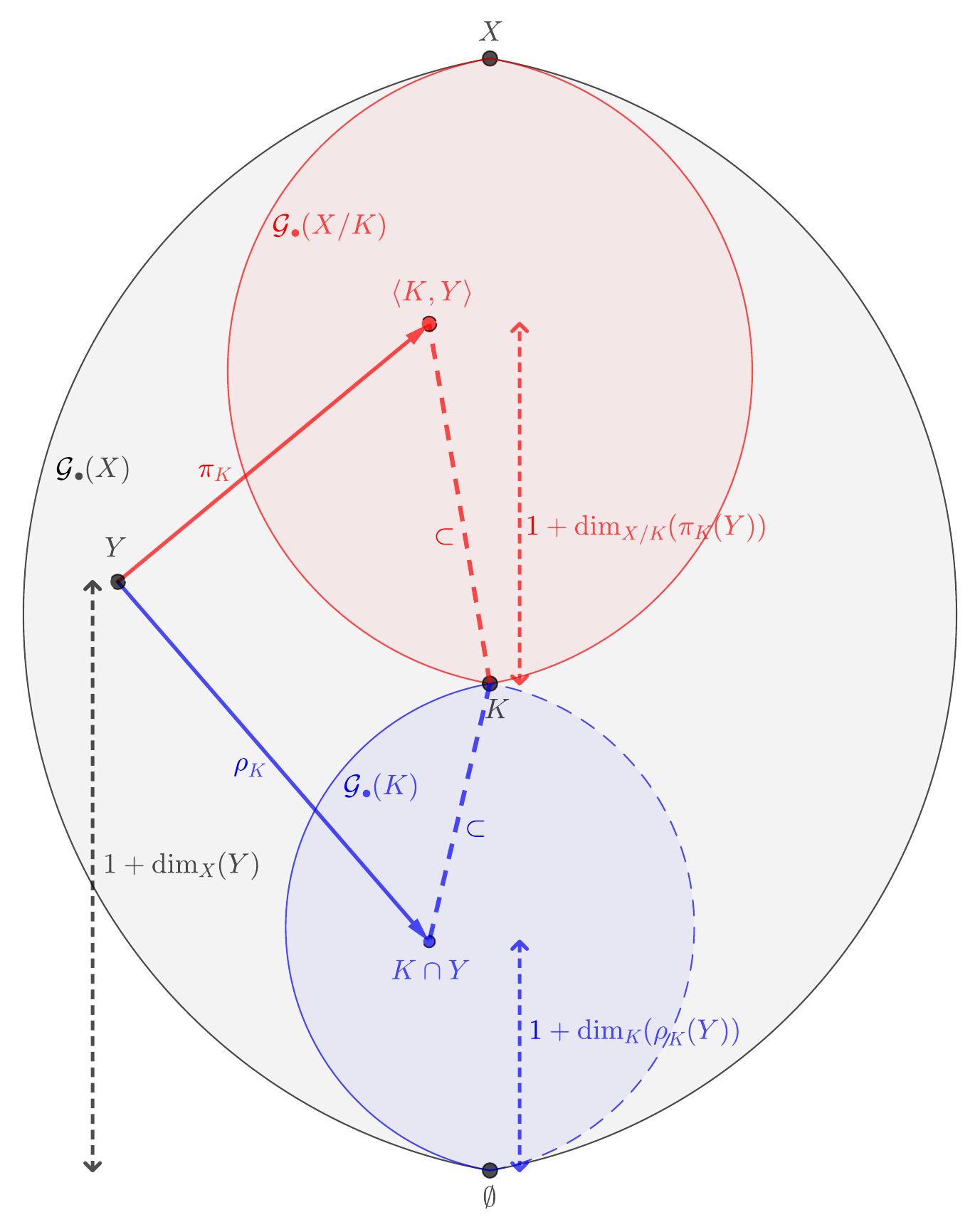}
    \caption{The lattice $\GrassmannSpace{X}$ of all subspaces of $X$ (see \autoref{def:grassmann}). 
    For a fixed $K\in\GrassmannSpace{X}$, 
    elements smaller than $K$ can be identified with $\GrassmannSpace{K}$ (blue part), 
    while elements bigger than $K$ can be identified with  $\GrassmannSpace{X/K}$ (red part) (see \autoref{def:quotientGeometry}).
    The maps $\pi_K$ (mapping to the quotient space $X/K$) and 
    $\rho_K$ (mapping to the subspace $K$) are indicated (see \autoref{def:shortExactSequence}). 
    In fact, $\generated{K,Y}$ is the least upper bound of $K$ and $Y$, 
    and $K\cap Y$ is the greatest lower bound of $K$ and $Y$ in the lattice $\GrassmannSpace{X}$. 
    The dimensional formula of \eqref{eq:dimension_Pi_Rho} translates to the fact that the length of the vertical segment on the left (black) equals the sum of the lengths of the vertical segments on the right (blue and red) -- assuming that subspaces of the same dimension are represented by points (of the lattice above) lying on a horizontal line, and that these lines are equidistant and are increasing with the dimension.
    }
    
    \label{fig:maps}
\end{figure}

\section{ Upper bound for \texorpdfstring{$f(n,q)$}{f(n,q)}: constructions}\label{sect:constr}

Here we first show a recursive construction scheme which gives $(s,t)$-blocking sets in dimension $n$, relying on constructions (or upper bounds) in  projective spaces of smaller dimension. Next, we specialise to the case $(s,t)=(2,1)$.
Then we will prove which recursive steps will lead to the best upper bounds in the general scheme, in arbitrary dimension.

\subsection{Recursive construction for general \texorpdfstring{$(s,t)$}{(s,t)}}\label{sec:3.1}

Our aim is to construct $(s,t)$-blocking sets recursively by passing to the quotient space. For this, we will need the following natural maps between the space, its subspace and its quotient.

\begin{defn}\label{def:shortExactSequence}
    Let $X$ be a projective space and $K\in\GrassmannDimSpace{k}{X}$. 
    Define the following inclusion-preserving maps  
    \[\begin{tikzcd}
        \GrassmannSpace{K} \ar[r,->,hook,"\iota_K"] 
        & \GrassmannSpace{X} \ar[r,->>,"\pi_K"]  \ar[l,->>,bend left,"\rho_K"]
        & \GrassmannSpace{X/K}  \ar[l,->,hook,bend left,"\theta_K"]
    \end{tikzcd}\]
    between the set of subspaces by
    $\iota_K\from Y\mapsto Y$ (the natural inclusion),  
    $\pi_K\from Y\to \generated{K,Y}$ (the natural projection), 
    $\rho_K\from Y\mapsto K\cap Y$ and
    $\theta_K\from Z\mapsto Z$. 
    See \autoref{fig:maps}.
\end{defn}
\begin{remark}\label{rem:dimProjection}
    Note that the compositions $\rho_K\circ \iota_K=\id_{\GrassmannSpace{K}}$ and 
    $\pi_K\circ \theta_K = \id_{\GrassmannSpace{X/K}}$ are both the identity maps, while the compositions    
    $\pi_K\circ \iota_K\from \GrassmannSpace{K}\to \GrassmannDimSpace{-1}{X/K}=\{K\}$ and 
    $\rho_K\circ \theta_K\from \GrassmannSpace{X/K}\to \GrassmannDimSpace{k}{K}=\{K\}$
    are both the constant maps. 
    Note that $\iota_K(Y)\subseteq K\subseteq \theta_K(Z)$ for every $Y\in\GrassmannSpace{K}$ and $Z\in\GrassmannSpace{X/K}$ by \autoref{def:quotientGeometry}.
    Furthermore,  for any $Y\in\GrassmannSpace{X}$, we have 
    \begin{equation}
        \dim_{K}(\rho_K(Y)) + \dim_{X/K}(\pi_K(Y)) = \dim_X(Y) -1,
        \label{eq:dimension_Pi_Rho}
    \end{equation}
    whereas the other two maps restrict to $\theta\from\GrassmannDimSpace{r}{X/K}\to \GrassmannDimSpace{r+k+1}{X}$ (for any $-1\leq r\leq n-k-1$), 
    and to $\iota_K\from \GrassmannDimSpace{d}{K}\to \GrassmannDimSpace{d}{X}$ (for any $-1\leq d\leq k$).
\end{remark}

\begin{figure}[!htb]
    \centering
    \includegraphics[height=0.5\textheight]{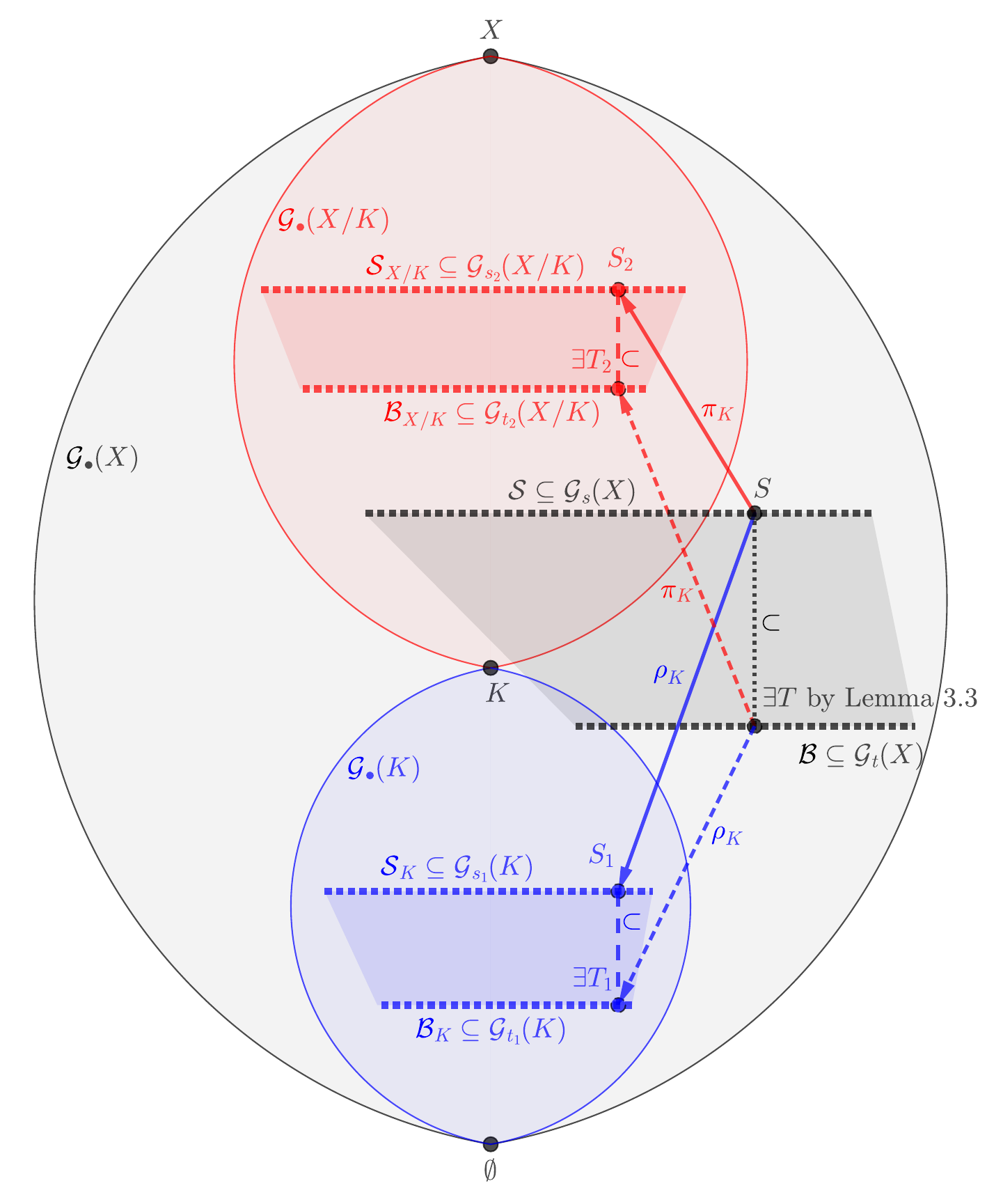}
    \caption{The main idea of the proof of \autoref{partial} is indicated in this figure. 
    Start from a subspace $S\in\cS$ to be blocked. 
    Then $S_1\leteq\rho_K(S)\in\cS_{K}$ is blocked by some $T_1\in\blocking_{K}$, 
    and $S_2\leteq\pi_K(S)\in\cS_{X/K}$ is blocked by some $T_2\in\blocking_{X/K}$. 
    By \autoref{lem:T1T2}, there exists $T\in\GrassmannSpace{X}$ such that $\rho_K(T)=T_1$ and $\pi_K(T)=T_2$. 
    Finally, one can show that $T$ is actually from $\blocking$ and blocks $S$. 
    (Here, thick horizontal dashed lines indicate subsets consisting of subspaces of given dimension. 
    Shaded regions between these lines indicate the connection between partial blocking sets and the sets they block.) 
    }
    \label{fig:pullback}
\end{figure}

The isomorphism of $U\oplus (V/U)\isom V$ for vector spaces implies the next statement, cf. \autoref{defn:Proj}.
\begin{lemma}[Pullback]\label{lem:T1T2}
    Let $K,N\in\GrassmannSpace{X}$ be  complements, i.e. $\generated{K,N}=X$ and $K\cap N=\emptyset$. 
    Then for any $T_1\in \GrassmannSpace{K}$ and $T_2\in \GrassmannSpace{X/K}$, the subspace $T\leteq \delta_{K,N}(T_1,T_2)\leteq \generated{\iota_K(T_1),\theta_K(T_2)\cap N}\in\GrassmannSpace{X}$ satisfies $\rho_K(T)=T_1$ and $\pi_K(T)=T_2$. 
    In other words, 
    the composition \[\begin{tikzcd}[column sep=large]
        \GrassmannSpace{K}\times \GrassmannSpace{X/K} \ar[r,hook,"\delta_{K,N}"] & 
        \GrassmannSpace{X} \ar[r,->>,"{(\rho_K,\pi_K)}"] &
        \GrassmannSpace{K}\times \GrassmannSpace{X/K}
    \end{tikzcd}\]
    is the identity map on $\GrassmannSpace{K}\times \GrassmannSpace{X/K}$. 
\end{lemma}
\begin{proof}
    Write $X=\Projectivisation{V}$. Then $V=\V{K}\oplus\V{N}$ by the definition of $N$. 
    Define $U_1\leteq \V{\iota_K(T_1)}\leq \V{K}$ and 
    $U_2\leteq \V{\theta_K(T_2)\cap N} \leq \V{N}$. 
    Note that $\delta_{K,N}(T_1,T_2)=\Projectivisation{U_1\oplus U_2}$.
    
    Then for the first part of the statement, we have 
    \[
        \iota_K(\rho_K(\delta_{K,N}(T_1,T_2)))=
        \iota_K(\rho_K(\Projectivisation{U_1\oplus U_2})) = \Projectivisation{\V{K}\cap (U_1\oplus U_2}) = \Projectivisation{U_1}=\iota_K(T_1).\]
    So cancelling $\iota_K$ (by applying $\rho_K$) we get $\rho_K(\delta_{K,N}(T_1,T_2))=T_1$ as needed.
        
    Next we claim that for any $U\geq \V{K}$, we have 
    $\V{K}\oplus ( U \cap \V{N})   = U$. Indeed, $\V{K}\oplus ( U \cap \V{N})   \subseteq U$ follows from the assumption. For the other containment, if $u\in U$, then $u=u_K+u_N$ for some (unique) $u_K\in \V{K}$, $u_N\in\V{N}$, so 
    $u_N=u-u_K\in U$, i.e. $\V{K}\oplus ( U \cap \V{N})   \supseteq U$.
    
    For the other part, we have  
    \begin{align*}
        \theta_K(\pi_K(\delta_{K,N}(T_1,T_2))) &=
        \theta_K(\pi_K(\Projectivisation{U_1\oplus U_2})) = 
        \Projectivisation{\generated{\V{K}, (U_1\oplus U_2}}) = 
        \Projectivisation{\V{K}\oplus U_2} 
        \\&= \Projectivisation{\V{K}\oplus ( \V{\theta_K(T_2)}\cap \V{N})}   = 
        \Projectivisation{\V{\theta_K(T_2)}} = \theta_K(T_2)
    \end{align*}
    using the claim for $U=\V{\theta_K(T_2)}$. 
    Cancelling $\theta_K$ (by applying $\pi_K)$ proves the  statement.
\end{proof}

We extend the notion of $(s,t)$-blocking sets.
\begin{defn}[Partial blocking set]
    For $t\leq s\leq \dim(X)$, 
    we say $\blocking\subseteq \GrassmannDimSpace{t}{X}$ is a \emph{partial blocking set} for $\cS\subseteq \GrassmannDimSpace{s}{X}$ if for every $S\in\mathcal
    S$ there exists $T\in \blocking$ such that $T\subseteq S$. 
    In this case, we also say $\blocking$ \emph{blocks} $\cS$. Note that if $\blocking\subseteq \GrassmannDimSpace{t}{X}$ blocks $\cS\leteq \GrassmannDimSpace{s}{X}$, 
    then $\blocking$ is an $(s,t)$-blocking set in $X$.
\end{defn}

The next statement is fundamental in the paper on which all constructions are built. We can construct partial blocking sets in the original projective space from partial blocking sets of a subspace and of the corresponding quotient space.
\begin{lemma}[Recursive construction of partial blocking sets]\label{partial}
    Let $X$ be an $n$-dimensional projective space, $K\in\GrassmannDimSpace{k}{X}$, and assume the integers $s_1,s_2,t_1,t_2$ satisfy $-1\leq t_1\leq s_1\leq \dim(K)$ and $-1\leq t_2\leq s_2\leq  \dim(X/K)$. 
    Write $t\leteq t_1+t_2+1$ and $s\leteq s_1+s_2+1$.
    
    If $\blocking_K\subseteq \GrassmannDimSpace{t_1}{K}$ blocks $\cS_{K}\subseteq\GrassmannDimSpace{s_1}{K}$ 
    and $\blocking_{X/K}\subseteq \GrassmannDimSpace{t_2}{X/K}$ blocks $\cS_{X/K}\subseteq\GrassmannDimSpace{s_2}{X/K}$, 
    then $\blocking\subseteq\GrassmannDimSpace{t}{X}$ blocks $\cS\subseteq\GrassmannDimSpace{s}{X}$
    where 
    \begin{align*}
        \blocking&\leteq 
        \setBuilder{T\in \GrassmannSpace{X}}{K\cap T\in\blocking_K,\generated{K,T}\in\blocking_{X/K}},
        \\
        \cS &\leteq 
        \setBuilder{S\in\GrassmannSpace{X}}{K\cap S\in\cS_K,\generated{K,S}\in\cS_{X/K}}  .
    \end{align*}
\end{lemma}

\begin{remark}\label{rem:partial}
    The conditions of \autoref{partial} imply that $-1\leq t\leq s\leq \dim(X)$, since $\dim(X/K)=n-k-1$. 
    The conditions on $k$ translate to $s_1\leq k\leq n-s_2-1$.

    Note that for any $T\in\blocking$, we have $\dim_K(K\cap T)=t_1$ and $\dim_{X/K}(\generated{K,T})=t_2$, 
    whereas 
    for any $S\in\cS$, we have $\dim_K(K\cap S)=s_1$ and $\dim_{X/K}(\generated{K,S})=s_2$. 
    Then \eqref{eq:dimension_Pi_Rho} from \autoref{rem:dimProjection} shows that we indeed have $\dim_X(T)=t$ and $\dim_X(S)=s$ using the definition of $t$ and $s$.

    From a more structural viewpoint, \autoref{partial} can be summarised using the following diagram (where all maps are surjective and vertical lines indicate partial blocking sets). 
   \[\begin{tikzcd}[row sep=tiny]
    \GrassmannDimSpace{s_1}{K} 
    & \GrassmannDimSpace{s}{X} \ar[r,->>, "\pi_K"] \ar[l,->>, "\rho_K"']
    & \GrassmannDimSpace{s_2}{X/K}
    \\
    \cS_K \ar[u, phantom, sloped, "\subseteq"]
    & \overbrace{\rho_K^{-1}(\cS_K)\cap \pi_K^{-1}(\cS_{X/K})}^{\cS} \ar[r,->>, "\pi_K"] \ar[l,->>, "\rho_K"'] \ar[u, phantom, sloped, "\subseteq"]
    & \cS_{X/K} \ar[u, phantom, sloped, "\subseteq"] 
    \\[10pt]
    \blocking_K \ar[d, phantom, sloped, "\subseteq"] \ar[u,-,"\text{partial}", "\text{blocking set}"']
    & \underbrace{\rho_K^{-1}(\blocking_K)\cap \pi_K^{-1}(\blocking_{X/K})}_{\blocking} \ar[r,->>, "\pi_K"] \ar[l,->>, "\rho_K"'] \ar[d, phantom, sloped, "\subseteq"] \ar[u,-,"\text{partial}", "\text{blocking set}"']
    & \blocking_{X/K} \ar[d
    , phantom, sloped, "\subseteq"] \ar[u,-,"\text{partial}", "\text{blocking set}"']
    \\
    \GrassmannDimSpace{t_1}{K} 
    & \GrassmannDimSpace{t}{X} \ar[r,->>, "\pi_K"] \ar[l,->>, "\rho_K"']
    & \GrassmannDimSpace{t_2}{X/K}
    \end{tikzcd}\]

    In this paper, we mostly focus on the following special cases, where the formulae from \autoref{partial}  simplify to  
    \begin{align*}
    \blocking &=
    \begin{cases}
        \setBuilder{T\in\GrassmannDimSpace{t}{X}}{\generated{K,T}\in\blocking_{X/K}}
        &\text{if $t_1=-1$ (i.e. $t_2=t$),} 
        \\
        \setBuilder{T\in\GrassmannDimSpace{t}{X}}{
        \generated{K,T}\in\blocking_{X/K}}
        &\text{if $t_1=s_1$,}
        \\  
        \blocking_K
        &\text{if $t_2=-1$ (i.e. $t_1=t$),} 
        \\  
        \setBuilder{T\in \GrassmannDimSpace{t}{X}}{
        K\cap T\in\blocking_K}
        &\text{if $t_2=s_2$,}
    \end{cases}
    \\
    \cS &= \GrassmannDimSpace{s}{X}\text{ \qquad if $\cS_K=\GrassmannDimSpace{s_1}{K}$ and $\cS_{X/K}=\GrassmannDimSpace{s_2}{X/K}$.}
    \end{align*}
\end{remark}

\begin{proof}[Proof of \autoref{partial}]
    The idea of the proof is demonstrated in \autoref{fig:pullback}.
    Note that $\blocking\subseteq \GrassmannDimSpace{t}{X}$ and $\cS\subseteq \GrassmannDimSpace{s}{X}$ by \autoref{rem:partial}.

    Pick $S\in \cS$. 
    Then on one hand, by definition $S_1\leteq \rho_K(S)\in \cS_{K}$, so 
    there is $T_1\in\blocking_K\subseteq \GrassmannDimSpace{t_1}{K}$ such that $T_1\subseteq S_1$ by assumption. 
    Similarly, on the other hand, $S_2\leteq \pi_K(S)\in \cS_{X/K}$ by definition, 
    so by assumption, there is $T_2\in \blocking_{X/K}\subseteq \GrassmannDimSpace{t_2}{X/K}$ such that $T_2\subseteq S_2$. 
    Pick a complement $N$ of $K$ such that $\generated{K,S}\cap N\subseteq S$ (e.g. by picking maximal subspaces $N_1$ in $S\setminus K$ and $N_2$ in $X\setminus\generated{K,S}$, then take $N\leteq \generated{N_1,N_2}$). 
    
    Consider $T\leteq \delta_{K,N}(T_1,T_2)$ given by \autoref{lem:T1T2}.
    Then $\rho_K(T)=T_1\in\blocking_K$ and $\pi_K(T)=T_2\in\blocking_{X/K}$, 
    hence $T\in \blocking$ by definition. 
    Finally note that since $\delta_{K,N}$ is inclusion preserving, we have $T=\delta_{K,N}(T_1,T_2)\subseteq \delta_{K,N}(S_1,S_2) = 
    \generated{K\cap S,\generated{K,S}\cap N}\subseteq S$ by the choice of $N$. 
\end{proof}

We can obtain an $(s,t)$-blocking set in an $n$-dimensional space from blocking sets corresponding to smaller parameters.
\begin{cor}[Recursive constructions]\label{cor:recursionSimple}
    Let $X$ be an $n$-dimensional projective space, $-1\leq t\leq s\leq n$. 
    Pick $K\in\GrassmannDimSpace{k}{X}$ where $-1\leq k\leq n-s-1$. 
    For every $-1\leq d\leq \min\{k,s\}$, 
    suppose there are integers satisfying 
    $-1\leq t_1(d)\leq s_1(d)=d$, 
    $-1\leq t_2(d)\leq s_2(d)$
    and $t_1(d)+t_2(d)+1=t$ and $s_1(d)+s_2(d)+1=s$, 
    and  
    choose a blocking set $\blocking_K(d)\subseteq \GrassmannDimSpace{t_1(d)}{K}$ of  $\GrassmannDimSpace{s_1(d)}{K}$  
    and a blocking set $\blocking_{X/K}(d)\subseteq \GrassmannDimSpace{t_2(d)}{X/K}$ of   $\GrassmannDimSpace{s_2(d)}{X/K}$. 
    Then \[\blocking\leteq \bigcup_{d=-1}^{\min\{k,s\}} \blocking(d)\] 
    is an $(s,t)$-blocking set in $X$, 
    where 
    
    \begin{equation*}
        \begin{split}
             \blocking(d)\leteq & \rho_K^{-1}(\blocking_K(d))\cap \pi_K^{-1}(\blocking_{X/K}(d))
    \\=& \setBuilder{T\in \GrassmannDimSpace{t}{X}}{K\cap T\in\blocking_K(d),\generated{K,T}\in\blocking_{X/K}(d)}
        \end{split}
    \end{equation*}   
 is from \autoref{partial}.
\end{cor}
\begin{proof}
    By assumption, $s_1(d)=d\leq \min\{k,s\}\leq k=\dim(K)$, 
    and $s_2(d)=s-1-d\leq s\leq n-k-1=\dim(X/K)$, 
    so we may indeed apply \autoref{partial} for every $d$.
    
    Now pick an arbitrary $S\in\GrassmannDimSpace{s}{X}$. 
    Then $-1\leq d\leteq \dim_X(K\cap S)\leq \min\{\dim_X(K),\dim_X(S)\}=\min\{k,s\}$, so 
    there is $T\in\blocking(d)\subseteq \blocking$ such that $T\subseteq S$ by \autoref{partial} and \autoref{rem:partial}.
\end{proof}

In the following subsections, we will use \autoref{cor:recursionSimple} from above in the special case $(s,t)=(2,1)$. 
However, it is worth mentioning that the conditions of \autoref{cor:recursionSimple} can be satisfied with the following simple observation.

\begin{theorem}[An explicit basic recursive construction]
\label{cor:explicitRecusrionST}
    Let $-1\leq t\leq s\leq n$ and $-1\leq k\leq n-s-1$. 
    Let $X$ be an $n$-dimensional projective space and $K\in\GrassmannDimSpace{k}{X}$. 
    For $-1\leq d\leq \min\{k,t\}$, let  $\blocking_d\subseteq \GrassmannDimSpace{t-d-1}{X}$ be an $(s-d-1,t-d-1)$-blocking set in $X/K$. 
    Then the disjoint union 
    \[\blocking \leteq \bigcup_{d=-1}^{\min\{k,t\}}
    \setBuilder{T\in\GrassmannDimSpace{t}{X}}{\dim(K\cap T)=d,\generated{K,T}\in\blocking_d}
    \subseteq \GrassmannDimSpace{t}{X}\]
    is an $(s,t)$-blocking set in $X$.
\end{theorem}
\begin{remark}
    If $d=t$, then the only $(s-t-1,-1)$-blocking set in $X/K$ is $\blocking_t=\{K\}=\GrassmannDimSpace{-1}{X/K}$. 
    Note that in order to use induction, we need $\dim(X/K)<\dim(X)$, i.e., we need $0\leq k$. 
\end{remark}

\begin{proof}[Proof of \autoref{cor:explicitRecusrionST}]
    Consider the variables defined by \autoref{tab:tisi}.
    \begin{table}[!htb]
    \centering
    \begin{tabular}{c||c|c|c|c|c|c}
         & $t_1(d)$ & $s_1(d)$ & $t_2(d)$ & $s_2(d)$ & $\blocking_K(d)$  & $\blocking_{X/K}(d)$
         \\\hline\hline
        $-1\leq d\leq \min\{k,t\}$ & $d$ & $d$ & $t-d-1$ & $s-d-1$ & $\GrassmannDimSpace{d}{K}$ & $\blocking_d$
        \\\hline
        $\min\{k,t\}< d\leq \min\{k,s\}$ & $t$ & $d$ & $-1$ & $s-d-1$ & $\GrassmannDimSpace{t}{K}$ & $\{K\}$
    \end{tabular}
    \caption{An explicit choice for $t_i(d)$, $s_i(d)$}
    \label{tab:tisi}
    \end{table}
   
    One can check that these values satisfy the conditions of  \autoref{cor:recursionSimple}, hence we obtain an $(s,t)$-blocking set $\blocking\leteq \bigcup_{d=-1}^{\min\{k,s\}} \blocking(d)$ in $X$. 
    Now \autoref{partial} and \autoref{rem:partial} give 
    \[\blocking(d)=\begin{cases}
        \setBuilder{T\in\GrassmannDimSpace{t}{X}}{
        \generated{K,T}\in\blocking_d},&\text{if $-1\leq d\leq \min\{k,t\}$}
        \\
        \blocking_K(d)=\GrassmannDimSpace{t}{K},&\text{if $\min\{k,t\}< d\leq \min\{k,s\}$}.
    \end{cases}\]
    Since $\blocking_t=\{K\}$, we have $\blocking(t)=\GrassmannDimSpace{t}{K}=\blocking(d)$ for every $d>t$. 
    Thus $\blocking=\bigcup_{d=-1}^{\min\{k,s\}} \blocking(d)=\bigcup_{d=-1}^{\min\{k,t\}} \blocking(d)$ is as stated.
\end{proof}

\begin{remark}
    
    There are many other ways of applying \autoref{partial} to get an $(s,t)$-blocking set. 
    Some of these choices may be better than the one in \autoref{cor:explicitRecusrionST}, cf. \autoref{sec:construction21}.
\end{remark}

\subsection{Special case of \texorpdfstring{$(s,t)=(2,1)$}{(s,t)=(2,1)}}\label{sec:3.2}

In this section, we analyse \autoref{cor:explicitRecusrionST} in case of $(s,t)=(2,1)$.

The following simple construction is obvious, we add it here for reasons of comparison, and also since it serves as a base case for a recursion.

\begin{lemma}[Trivial construction]\label{lem:trivialConstruction}
    Let $X$ be an $n$-dimensional projective space, and $H\in\GrassmannDimSpace{n-1}{X}$ be an arbitrary hyperplane. 
    Then $\blocking=\GrassmannDimSpace{1}{H}$ is a $(2,1)$-blocking set in $X$. 
    In particular, \[f(n,q) \leq \grassmannDimSpace{1}{H}=\qbinom{n}{2}{q}
    = \frac{(q^n-1)(q^{n-1}-1)}{(q^2-1)(q-1)}\]
    For $n\geq 5$, this gives 
    \[f(n,q)\leq \coefficients{1}{1}{2}{2}\]
\end{lemma}
\begin{proof}
    Pick $S\in\GrassmannDimSpace{2}{X}$. 
    Note that $\dim(H\cap S)=\dim(H)+\dim(S)-\dim(\generated{H,S})\geq (n-1)+2-n=1$
    as $\dim(\generated{H,S})\leq \dim(X)=n$.
    Thus there is a line $T\subseteq H\cap S$ which is then necessarily an element of $\blocking$ by definition as required.
\end{proof}

Now consider the recursive construction of \autoref{cor:explicitRecusrionST} for $(s,t)=(2,1)$. To be able to use it, we state the following known result for smaller parameters.

\begin{lemma}\label{lem:minimalBlockingForSmallST}
    Let $\blocking$ be a minimal $(s,t)$-blocking set in an $n$-dimensional projective space $X$. 
    Then 

    \[
    \blocking\in 
    \begin{cases}
        \setBuilder{\{T\}}{T\in\GrassmannDimSpace{1}{X}}, &\text{if $n=2$, $(s,t)=(2,1)$}
        \\
        \setBuilder{\mathcal{L}}{\text{$\mathcal{L}$ is a $1$-spread in $X$}},&\text{if $n=3$, $(s,t)=(2,1)$}
        \\
        \setBuilder{\GrassmannDimSpace{0}{H}}{H\in\GrassmannDimSpace{n-s}{X}},&\text{if $t=0$}
        \\
        \{\GrassmannDimSpace{t}{X}\} ,&\text{if $t\in\{-1,s\}$}
    \end{cases}\]
    and all such choices give minimal blocking sets. 
    In particular, $f(2,q)=1$ and $f(3,q)=q^2+1$, cf. \autoref{thm:f234}.
\end{lemma}
\begin{proof}
    The first and last cases are evident. 
    The second case follows from \autoref{thm:beutel}, 
    while the third case from \cite{BB}.
\end{proof}

Since there are too many parameters to analyse in \autoref{cor:recursionSimple}, we focus on the case $(s,t)=(2,1)$.  
Now \autoref{cor:explicitRecusrionST} specialises to the following construction that can be regarded as a generalisation of \autoref{lem:trivialConstruction}.

\begin{prop}[Basic recursion for $(s,t)=(2,1)$]\label{lem:constructionBasic21}
    Let $X$ be an $n$-dimensional projective space, $K\in\GrassmannDimSpace{k}{X}$ for some $-1\leq k\leq n-3$. 
    Let $\blocking_{-1}$ be a $(2,1)$-blocking set in $X/K$, 
    and let $H\in\GrassmannDimSpace{n-1}{X}$ be an arbitrary hyperplane containing $K$. 
    Then the disjoint union 
    \[\blocking= \setBuilder{T\in\GrassmannDimSpace{1}{X}}{\generated{K,T}\in\blocking_{-1}}
    \cup \setBuilder{T\in\GrassmannDimSpace{1}{H}}{K\cap T\neq\emptyset} 
    \subseteq \GrassmannDimSpace{1}{X}\]    
    is a $(2,1)$-blocking set in $X$.
\end{prop}
\begin{proof}
    Note that $\pi_K(H)\in\GrassmannDimSpace{n-k-2}{X/K}$ from \autoref{def:shortExactSequence} is a hyperplane in $X/K$, so \autoref{lem:minimalBlockingForSmallST} implies that $\blocking_0\leteq \GrassmannDimSpace{0}{\pi_K(H)}$ is a (minimal) $(1,0)$-blocking set in $X/K$.
    Also, $\blocking_1\leteq \{K\}$ is a $(0,-1)$-blocking set in $X/K$. 
    Then, for $\blocking(d)\leteq \setBuilder{T\in\GrassmannDimSpace{t}{X}}{\dim(K\cap T)=d,\generated{K,T}\in\blocking_d}$, $\blocking \leteq \bigcup_{d=-1}^{\min\{k,t\}} \blocking(d)$ is a $(2,1)$-blocking set by \autoref{cor:explicitRecusrionST}. 
    We determine those sets. 
    First note that 
    \begin{align*}
        \blocking(-1)&=\setBuilder{T\in\GrassmannDimSpace{1}{X}}{\generated{K,T}\in\blocking_{-1}}    
    \intertext{by \autoref{partial} and \autoref{rem:partial}.
    Next, }
        \blocking(0) &= \setBuilder{T\in\GrassmannDimSpace{1}{X}}{\dim(K\cap T)=0,\generated{K,T}\in\GrassmannDimSpace{0}{\pi_K(H)}} 
        \\&= \setBuilder{T\in\GrassmannDimSpace{1}{X}}{\dim(K\cap T)=0,\generated{K,T}\in\GrassmannDimSpace{k+1}{H}} 
        \\&=
    \setBuilder{T\in\GrassmannDimSpace{1}{H}}{\dim(K\cap T)=0},
    \intertext{because $\GrassmannDimSpace{0}{\pi_K(H)}=\setBuilder{Y\in\GrassmannDimSpace{k+1}{H}}{K\subset Y}$
    and $\dim(\generated{K,T})=\dim(K)+\dim(T)-\dim(K\cap T)$. 
    Finally, we have }
        \blocking(1)&=\setBuilder{T\in\GrassmannDimSpace{1}{X}}{\dim(K\cap T)=1,\generated{K,T}\in\{K\}}
        \\&=\setBuilder{T\in\GrassmannDimSpace{1}{H}}{\dim(K\cap T)=1}.
    \end{align*}
    Hence $\blocking(0)\cup\blocking(1)=\setBuilder{T\in\GrassmannDimSpace{1}{H}}{K\cap T\neq \emptyset}$    and the statement follows.
\end{proof}

As a special case, we recover the construction of \autoref{lem:trivialConstruction}.
\begin{remark}\label{rem:constructionRecursiveTrivial}
    We claim that in the case $\dim(X/K)=2$ (i.e. $k=n-3$) and $\blocking_{-1}\leteq \{H\}$, \autoref{lem:constructionBasic21}  gives $\blocking=\GrassmannDimSpace{1}{H}$, which is exactly the construction of \autoref{lem:trivialConstruction}.
    Indeed, using \autoref{rem:dimProjection}, we have $H\in\GrassmannDimSpace{n-k-2}{X/K}=\GrassmannDimSpace{1}{X/K}$. 
    Thus \autoref{lem:minimalBlockingForSmallST} shows that we can indeed take $\blocking_{-1}=\{H\}$ because $\dim(X/K)=2$.
    Then $\blocking(-1)\leteq \setBuilder{T\in\GrassmannDimSpace{1}{X}}{\generated{K,T}=H} = \setBuilder{T\in\GrassmannDimSpace{1}{H}}{K\cap T=\emptyset}$, so we are done.
\end{remark}

We can attempt to use \autoref{lem:constructionBasic21} recursively to find a smaller construction, however, it turns out that this gives no improvement.
\begin{remark}\label{rem:basinConstructionStepInvariant}
    Let $\blockingSets{2}{1}{X}$ denote the set of all $(2,1)$-blocking sets in $X$ and fix  a  hyperplane $H\in\GrassmannDimSpace{n-1}{X}$. 
    Then there is a map $\blockingSets{2}{1}{X/K}\xrightarrow{H} \blockingSets{2}{1}{X}$ given by  $\blocking_{-1}\mapsto \blocking$.    

    For a sequence $K_r \subseteq \dots \subseteq K_2\subseteq K_1\subseteq X$ of subspaces, the resulting composition 
    $\blockingSets{2}{1}{X/K_1} \xrightarrow{H}\blockingSets{2}{1}{X/K_2} \xrightarrow{H}\dots \xrightarrow{H}\blockingSets{2}{1}{X/K_r}\xrightarrow{H}\blockingSets{2}{1}{X}$ 
    equals the map $\blockingSets{2}{1}{X/K_1}\xrightarrow{H} \blockingSets{2}{1}{X}$. 
    Thus, for the resulting blocking set, only the starting blocking set in $X/K$ is relevant, and not the number of iterations.

    It would be interesting to see if an analogous statement holds in general for the construction of \autoref{cor:explicitRecusrionST}.
\end{remark}

By picking a different starting dimension, we can improve the construction. 
To see the size of the improvement, we need the following computational statements.

\begin{lemma}\label{lem:countingComplementarySubspaces}
    Let $U\leq V$ be vector spaces. 
    Then there is a non-canonical bijection between linear maps $V/U\to U$ and complementary subspaces of $U$ in $V$.
\end{lemma}
\begin{proof}
    Write $V=U\oplus W_0$ for some $W_0\leq V$. By picking a non-canonical isomorphism $W_0\to V/U$, it is enough to parametrise complementary subspaces by linear maps $f\from W_0\to U$. 
    For such a linear map $f$, we assign the complementary subspace $W_f\leteq \setBuilder{w_0-f(w_0)}{w_0\in W_0}$. 
    Conversely, for a complement $W$ of $U$, we assign the linear map $f_W\from W_0\embeds U\oplus W\surjects U$.
    These assignments are inverses of each other.
\end{proof}

\begin{lemma}\label{lem:constructionFibreSize}
    Let $X=\PG(n,q)$, $K\in\GrassmannDimSpace{k}{X}$ and $\blocking\subseteq\GrassmannDimSpace{t}{X/K}$ where $-1\leq k\leq n$ and $-1\leq t\leq \dim(X/K)=n-k-1$. 
    Then $\cardinality{\setBuilder{T\in\GrassmannDimSpace{t}{X}}{\generated{K,T}\in\blocking}} 
    = q^{(t+1)(k+1)}\cardinality{\blocking}$.
\end{lemma}
\begin{proof}
    Write $X=\Projectivisation{V}$ for some $(n+1)$-dimensional vector space $V$. Then the following computation gives the statement.
    \begin{align*}
        \cardinality{\setBuilder{T\in\GrassmannDimSpace{t}{X}}{\generated{K,T}\in\blocking}} &=
        \sum_{T_2\in \blocking} \cardinality{\setBuilder{T\in\GrassmannDimSpace{t}{X}}{K\cap T=\emptyset, \generated{K,T}=T_2}} \\&=
        \sum_{T_2\in \blocking} \cardinality{\setBuilder{U\leq V}{\V{T_2}=\V{K}\oplus U}} \\&=
        \sum_{T_2\in \blocking} \cardinality{\setBuilder{\alpha\from \V{T_2}/\V{K}\to \V{K}}{\text{$\alpha$ is linear}}} \\&=
        \cardinality{\blocking}q^{t+1}q^{k+1}
    \end{align*}
    Here we used the following. 
    First,  $K\cap T=\emptyset$ is automatically satisfied, because $\dim_X(K\cap T) = \dim_X(T)-\dim_{X/K}(\generated{K,T})-1 = -1$ by
    \autoref{rem:dimProjection}. 
    Second, we wrote $\V{T}=U\leq V$. 
    Third, we used \autoref{lem:countingComplementarySubspaces}.
    Finally, we used $\Dim(\V{T_2}/\V{K})=t+1$, $\Dim(\V{K})=k+1$.
    \end{proof}

Now we are ready to compute the improvement (compared to \autoref{rem:constructionRecursiveTrivial}) if we start from the $3$-dimensional case instead.
\begin{remark}\label{rem:construction1step}
    Let $K_2\subseteq K_1\subseteq X$ with $\dim(K_2)=n-4$ and $\dim(K_1)=n-3$. 
    Let $\mathcal{L}$ be a $1$-spread in $X/K_2$ which is a $(2,1)$-blocking set by \autoref{lem:minimalBlockingForSmallST} as $\dim(X/K_2)=3$. 
    Applying the map of \autoref{rem:basinConstructionStepInvariant}, we obtain the following $(2,1)$-blocking sets.
    \[\begin{tikzcd}[row sep=tiny]
        \blockingSets{2}{1}{X/K_1}\ar[r,"H"] 
        & 
        \blockingSets{2}{1}{X/K_2}\ar[r,"H"] 
        &
        \blockingSets{2}{1}{X} 
        \\
        \{H\}\ar[r,|->] & 
        \blocking''\leteq \setBuilder{T\in\GrassmannDimSpace{1}{X/K_2}}{T\subseteq H} \ar[r,|->] & 
        \blocking'\leteq \setBuilder{T\in\GrassmannDimSpace{1}{X}}{T\subseteq H}
        \\
        & \mathcal{L} \ar[r,|->] 
        & \blocking
    \end{tikzcd}
    \]

    Assume $X=\PG(n,q)$. 
    Then $\cardinality{\blocking''}=\qbinom{3}{2}{q}$ and $\cardinality{\blocking'}=\qbinom{n}{2}{q}$ by \autoref{rem:constructionRecursiveTrivial}, 
    and $\cardinality{\mathcal{L}}=q^2+1$ by 
    \autoref{lem:minimalBlockingForSmallST}.
    Then \autoref{lem:constructionFibreSize} shows that 
    $\cardinality{\blocking'}-\cardinality{\blocking}=q^{2(\dim(K_2)+1)}\cdot (\cardinality{\blocking''}-\cardinality{\mathcal{L}}) = 
    q^{2n-5}$.  
    Since $\cardinality{\blocking'}=\qbinom{n}{2}{q}$ by \autoref{lem:trivialConstruction} and \autoref{rem:constructionRecursiveTrivial}, we have 
    \[\cardinality{\blocking} = \qbinom{n}{2}{q}-q^{2n-5}
    = \frac{(q^n-1)(q^{n-1}-1)}{(q^2-1)(q-1)}-q^{2n-5}.\]
    Of course, $\cardinality{\blocking}$ could be determined in a more straightforward way using more computations similar to what we will do in \autoref{recursion}.
    For $n\geq 5$, this gives 
    \[f(n,q)\leq \coefficients{1}{0}{2}{2}.\]
\end{remark}

\subsection{Improved recursion for \texorpdfstring{$(s,t)=(2,1)$}{(s,t)=(2,1)}}\label{sec:construction21}

Now we aim to improve the construction of \autoref{rem:construction1step}. 
Recall that it originated from \autoref{cor:explicitRecusrionST}, which in turn was an application of the more general \autoref{partial}. 
We will apply the latter result differently as follows.
Consider the setup of \autoref{tab:tisi}. We keep the $d=-1,2$ cases unchanged, but modify the $d=0$ case carefully so that the $d=1$ case could be omitted completely when applied to \autoref{cor:recursionSimple}.
The idea is to consider $X/P$ for $P\in K$ instead of the usual $X/K$, as it will give us enough freedom to have suitable control over the lines intersecting $K$ in a single point.

\begin{prop}[Improved recursion for $(s,t)=(2,1)$]\label{lem:construction21ImprovedGeneral}
    Let $X$ be an $n$-dimensional projective space and 
    $K\in\GrassmannDimSpace{k}{X}$ for $-1\leq k\leq n-3$.
    Let $\blocking_K\subseteq \GrassmannDimSpace{1}{K}$ be a $(2,1)$-blocking set in $K$, 
    $\blocking_{X/K}\subseteq \GrassmannDimSpace{1}{X/K}$ be a $(2,1)$-blocking set in $X/K$ and 
    $\blocking_{X/P}\subseteq\GrassmannDimSpace{0}{X/P}$ be a $(1,0)$-blocking set in $X/P$ for every $P\in \GrassmannDimSpace{0}{K}$. 
    Assume that for every $L\in\GrassmannDimSpace{1}{K}\setminus\blocking_K$, the set      
    $\bigcup\setBuilder{\blocking_{X/P}\setminus\GrassmannDimSpace{0}{K/P}}{P\in \GrassmannDimSpace{0}{L}}\subseteq \GrassmannDimSpace{1}{X}$ blocks 
    $\cS_L\leteq 
\setBuilder{S\in\GrassmannDimSpace{2}{X}}{K\cap S=L}$. 
    
    Then the disjoint union
    \[
    \blocking\leteq 
    \setBuilder{T\in\GrassmannDimSpace{1}{X}}{\generated{K,T}\in\blocking_{X/K}} \disjointunion 
    \Big(\disjointUnion_{{P\in \GrassmannDimSpace{0}{K}}} (\blocking_{X/P}\setminus\GrassmannDimSpace{0}{K/P})\Big) \disjointunion 
    \blocking_K
    \subseteq \GrassmannDimSpace{1}{X}\]
    is a $(2,1)$-blocking set in $X$.
\end{prop}

\begin{proof}
    Define the following parameters as in \autoref{tab:tisi21}.     Note that the $d\in\{-1,2\}$ cases match exactly those in \autoref{tab:tisi}, the $d=0$ case is replaced considering every $P\in \GrassmannDimSpace{0}{K}$ instead of $K$. The $d=1$ case is omitted.
    \begin{table}[!htb]
    \centering
    \begin{tabular}{c|c||c|c|c|c|c|c|c}
         $d$ & $Y$  & $t_1(d)$ & $s_1(d)$ & $t_2(d)$ & $s_2(d)$  & $K(Y)$  & $\blocking_{K(Y)}(d)$ & $\blocking_{X/K(Y)}(d)$ 
         \\\hline\hline
        $-1$ & $K$ & $-1$ & $-1$ & $1$ & $2$ & $K$ & $\{\emptyset\}$ & $\blocking_{X/K}$
        \\\hline
        $0$ & $P\in\GrassmannDimSpace{0}{K}$ & $0$ & $0$ & $0$ & $1$ & $P$ & $\{P\}$ & $\blocking_{X/P}$ 
        \\\hline
        $2$ & $K$ & $1$ & $2$ & $-1$ & $-1$ & $K$  & $\blocking_K$ & $\{K\}$
    \end{tabular}
    \caption{An explicit choice for $t_i(d)$, $s_i(d)$ for $d\in\{-1,0,2\}$, $(s,t)=(2,1)$}
    \label{tab:tisi21}
    \end{table}
    
    Apply \autoref{partial} with the setup of \autoref{tab:tisi21} for $d\leq \min\{2,k\}$ and call the resulting blocking set $\blocking(d,Y)$. 
    We have  $\blocking(-1,K)=\setBuilder{T\in\GrassmannDimSpace{1}{X}}{\generated{K,T}\in\blocking_{X/K}}$, 
    $\blocking(0,P)=
    \blocking_{X/P}$
    and $\blocking(2,K)=\blocking_K$
    by  \autoref{rem:partial}.

    We claim  that $\blocking = \blocking(-1,K)\cup \bigcup\setBuilder{ \blocking(0,P)\setminus\GrassmannDimSpace{1}{K}}{P\in \GrassmannDimSpace{0}{K}} \cup \blocking(2,K)$ is a $(2,1)$-blocking set as stated.
    Indeed, pick $S\in\GrassmannDimSpace{2}{X}$ and look for a $T\in\blocking$ with $T\subset S$ by considering cases for $d\leteq \dim(K\cap S)$. 
    If $d\in\{-1,2\}$, then by construction, there is $T\in\blocking(d,K)$ such that $T\subset S$. 
    If $d=0$, then $P\leteq K\cap S\in\GrassmannDimSpace{0}{K}$, so there is $T\in\blocking(0,P)$ such that $T\subset S$. Note that $T\notin\GrassmannDimSpace{0}{K/P}$ as $d\neq 1$.
    If $d=1$, then $L\leteq K\cap S\in\GrassmannDimSpace{1}{K}$. 
    If $L\in \blocking_K$, then we take $T=L\in\blocking_K=\blocking(2,K)$. 
    Otherwise, since $S\in\cS_L$, by assumption, there is $P\in\GrassmannDimSpace{0}{L}\subseteq \GrassmannDimSpace{0}{K}$ and $T\in \blocking_{X/P}\setminus\GrassmannDimSpace{0}{K/P}=\blocking(0,P)\setminus\GrassmannDimSpace{0}{K/P}$ such that $T\subset S$.

    The union is indeed disjoint, as the lines from different sets intersect $K$ in different subspaces.
\end{proof}

To show the existence of suitable blocking sets $\blocking_{X/P}$ for \autoref{lem:construction21ImprovedGeneral}, we review some facts about orthogonal complements and polarities. See \cite[section 6.1]{Cameron} for further details.

\begin{defn}
    Let $V$ be a vector space over $\field{}$. 
    Let $\beta\from V\times V\to \field{}$ be a symmetric non-degenerate bilinear form, i.e. 
    for every $u,v\in V$ we have $\beta(u,v)=\beta(v,u)$, 
    for every $u\in V$ we have that $\beta(u,V)=0$ implies $u=0$, 
    moreover $\beta$ is $\field{}$-linear in both variables.
    For every subspace $U\leq V$, define the $\beta$-orthogonal subspace $U^\perp\leteq U^{\perp_\beta}$ to be $ \setBuilder{v\in V}{\beta(U,v)=0}\leq V$.
\end{defn}

\begin{remark} 
    The operation of taking orthogonal subspaces is inclusion-reversing, i.e. for subspaces $U_1\leq U_2$, we have $U_1^\perp \geq U_2^\perp$.
    The non-degeneracy of $\beta$ implies that $\Dim_\field{}(U)+\Dim_\field{}(U^\perp)=\Dim_\field{}(V)$. 
    The symmetry of $\beta$ implies that $(U^\perp)^\perp = U$ for every subspace $U\leq V$.
\end{remark}

We can translate orthogonal complements to projective spaces to get polarities. Recall \autoref{defn:Proj}.
\begin{defn}\label{defn:polarity}
    Let $X=\Projectivisation{V}$ be an $n$-dimensional projective space. 
    Let $\beta$ be a symmetric non-degenerate bilinear form on $V$. 
    For every subspace $Y\in\GrassmannSpace{X}$, define $\polarity{\beta}(Y)\leteq \Projectivisation{\V{Y}^{\perp_\beta}}$.
\end{defn}

\begin{remark}
    The setup above gives rise to maps $\polarity{\beta}\from \GrassmannDimSpace{d}{X}\to \GrassmannDimSpace{n-d-1}{X}$ for every $-1\leq d\leq n$.
    Note that $\polarity{\beta}\circ \polarity{\beta}$ is the identity, i.e. $\polarity{\beta}\from \GrassmannSpace{X}\to\GrassmannSpace{X}$ is a \emph{polarity}.
\end{remark}

\begin{lemma}\label{lem:ballon}
    In the setup of \autoref{defn:polarity}, 
    for any $L\in \GrassmannDimSpace{1}{X}$, we have 
    $\bigcup_{P\in L} \polarity{\beta}(P)=X$.    
\end{lemma}
\begin{proof}
    Pick $Q\in X$ and fix a corresponding $v_Q\in\V{Q}\setminus\{0\}$.
    Define an $\field{}$-linear map $\phi\from \V{L}\to \field{},v\mapsto \beta(v,v_Q)$. 
    The isomorphism theorem implies that 
    $\Dim_\field{}(\ker(\phi)) = 
    \Dim_\field{}(\V{L})-\Dim_\field{}(\im(\phi))\geq 
    \Dim_\field{}(\V{L})-\Dim_\field{}(\field{}) =
    2-1=1$. 
    So there is $0\neq v_L\in \V{L}$ such that $\beta(v_L,v_Q)=0$, 
    thus $\beta(\field{}v_L,v_Q)=0$, 
    i.e. $v_Q\in (\field{}v_L)^{\perp_\beta}$, 
    which means that 
    $Q\in \Projectivisation{(\field{}v_L)^{\perp_\beta}}=\polarity{\beta}(P)$ 
    for $P\leteq \Projectivisation{\field{}v_L}\in L$. (Here $\field{}v$ denotes the vector space over $\field{}$ generated by the vector $v$.)
\end{proof}

\begin{lemma}\label{lem:bilinearMap}
    Let $V$ be a finite dimensional $\field{}$-vector space, and $U\leq V$ a subspace such that $\Dim(U)\leq \frac{1}{2}\Dim(V)$. 
    Then there is a non-degenerate symmetric bilinear form $\beta\from V\times V\to \field{}$ such that $\beta(U,U)=0$.
\end{lemma}
\begin{proof}
    Pick a basis $B$ for $V$ such that there is a subset $B_U\subseteq B$ which is a basis for $U$. 
    By assumption, $\cardinality{B_U}\leq \frac{1}{2}\cardinality{B}$, so there is a bijection    $\phi\from B\to B$ such that $\phi\circ \phi$ is the identity on $B$ (i.e. $\phi$ is an involution) and 
    $\phi(B_U)\cap B_U=\emptyset$.  
    Define 
    \[
        \beta\from B\times B\to \field{},\qquad
        (b,b')\mapsto \begin{cases}
        1,& b'=\phi(b)\\
        0,& b'\neq \phi(b)
    \end{cases}
    \]
    and extend it $\field{}$-linearly to $\beta\from V\times V\to \field{}$.
    Then $\beta$ is symmetric as $b'=\phi(b)$ if and only if $b=\phi(b')$. 
    By construction, $\beta(b,b')=0$ for $b,b'\in B_U$, hence $\beta(U,U)=0$. 
    Finally for non-degeneracy, assume that $\beta(v,V)=0$ for some $v=\sum_{b\in B} v_b$. Then $0=\beta(v,\phi(b))=v_b$ for every $b\in B$, hence $v=0$.
\end{proof}

\begin{cor}\label{cor:absolute}
    For every $n$-dimensional projective space $X$ and subspace $K\in\GrassmannDimSpace{k}{X}$ with $-1\leq k\leq \frac{n-1}{2}$, there exists a polarity $\polarity{\beta}$ such that $K\subseteq \polarity{\beta}(K)$. (In this case, $K$ is called an \emph{absolute} subspace.)
\end{cor}

\begin{proof}
    Write $X=\Projectivisation{V}$ and let $U\leteq \V{K}\leq V$. 
    Now $\Dim(V)=\dim(X)+1=n+1$ and 
    $\Dim(U)=\dim(K)+1=k+1$, 
    so $\Dim(U)\leq \frac{1}{2}\Dim(V)$ by assumption. 
    Apply \autoref{lem:bilinearMap} to get a non-degenerate symmetric bilinear form $\beta\from V\times V\to \field{}$ such that $\beta(U,U)=0$. 
    Then $U\subseteq U^{\perp_\beta}$, so after going back to the projective space, 
    we have $K=\Projectivisation{U}\subseteq \Projectivisation{U^{\perp_\beta}}=\polarity{\beta}(K)$ as stated.
\end{proof}

\begin{cor}\label{lem:ballonAbsolute}
    Let $X$ be an $n$-dimensional projective space over an arbitrary field, and pick $K\in\GrassmannDimSpace{k}{X}$ with $-1\leq k\leq \frac{n-1}{2}$. 
    Then there exists a map $\polarity{}\from \GrassmannDimSpace{0}{K}\to \GrassmannDimSpace{n-1}{X}$ such that 
    \begin{enumerate}
        \item for every $P\in \GrassmannDimSpace{0}{K}$, we have $K\subseteq \polarity{}(P)$, and 
        \item for every $L\in \GrassmannDimSpace{1}{K}$, we have $\bigcup \setBuilder{\polarity{}(P)}{P\in L}=X$.
    \end{enumerate}
\end{cor}

\begin{proof}
    The restriction of the polarity given by \autoref{cor:absolute} to $\GrassmannDimSpace{0}{K}$ is suitable because of \autoref{lem:ballon}.
    \end{proof}

We are now ready to construct the $(1,0)$-blocking sets as in \autoref{lem:construction21ImprovedGeneral}.
\begin{cor}\label{cor:ballonConstruction}
    Let $X$ be an $n$-dimensional projective space over an arbitrary field, and pick $K\in\GrassmannDimSpace{k}{X}$ with $-1\leq k\leq \frac{n-1}{2}$. 
    Then for every $P\in\GrassmannDimSpace{0}{K}$, there is a $(1,0)$-blocking set $\blocking_{X/P}=\GrassmannDimSpace{0}{H_P}\subseteq \GrassmannDimSpace{0}{X/P}$ in $X/P$ for suitable $H_P\in\GrassmannDimSpace{n-2}{X/P}$    
    such that 
    \begin{enumerate}
        \item $\GrassmannDimSpace{0}{K/P}\subseteq \blocking_{X/P}$ and 
        \item    for every $L\in\GrassmannDimSpace{1}{K}$, the set
    $\blocking_L\leteq \bigcup\setBuilder{\blocking_{X/P}\setminus\GrassmannDimSpace{0}{K/P}}{P\in \GrassmannDimSpace{0}{L}} \subseteq \GrassmannDimSpace{1}{X}$ blocks the set 
    $\cS_L\leteq \setBuilder{S\in\GrassmannDimSpace{2}{X}}{K\cap S=L}$.  
    \end{enumerate}
 
\end{cor}
\begin{proof}
    Let $\polarity{}\from \GrassmannDimSpace{0}{K}\to \GrassmannDimSpace{n-1}{X}$ be the map from \autoref{lem:ballonAbsolute}. 
    We show that $H_P\leteq \pi_P(\polarity{}(P))\in\GrassmannDimSpace{n-2}{X/P}$ satisfies the statement.
    We claim that the resulting $\blocking_{X/P}$ is an $(1,0)$-blocking set in $X/P$. 
    Indeed, if $S\in\GrassmannDimSpace{1}{X/P}$, then 
    $\dim_{X/P}(H_P\cap S) = \dim_{X/P}(H_P) + \dim_{X/P}(H_P\cap S) - \dim_{X/P}(\generated{H_P, S}) \geq 
    n-2 + 1 - \dim(X/P) = 0$, 
    so for any $T\in H_P\cap S$, we have $T\in\blocking_{X/P}$ with $T\subset S$.

    For the first part, note that $K\subseteq \polarity{}(P)$ implies that $K/P=\pi_P(K)\subseteq \pi_P(\polarity{}(P))$, hence $\GrassmannDimSpace{0}{K/P}\subseteq \blocking_{X/P}$ by definition.

    For the second part,  pick $L\in \GrassmannDimSpace{1}{K}$. 
    Let $S\in \cS_L$. 
    Then by definition, we can pick $Q\in S\setminus K$. 
    Since $\bigcup_{P\in L}\polarity{}(P)=X$, there is a $P\in L$ such that $Q\in \polarity{}(P)$. 
    As $P\in L$ and $Q\notin L$, we necessarily have $P\neq Q$, so $\dim_X(T)=1$, thus  $T\in\blocking_{X/P}$ by definition. 
    On the other hand, $T\not\subseteq K$, because $Q\in T$ and $Q\notin K$, so $T\notin \GrassmannDimSpace{0}{K/P}$. 
    Hence $T\in \blocking_L$ and $T\subseteq S$, so $\blocking_L$ indeed blocks $\cS$.
\end{proof}

\begin{theorem}[Improved recursive construction for $(s,t)=(2,1)$]\label{lem:recursiveConsruction}
    Let $X$ be an $n$-dimensional projective space over an arbitrary field. 
    Let $K\in\GrassmannDimSpace{k}{X}$ with $-1\leq k\leq \min\{\frac{n-1}{2},n-3\}$. 
    Let $\blocking_K\subseteq \GrassmannDimSpace{1}{X}$ be a $(2,1)$-blocking set in $K$, 
    $\blocking_{X/K}$ be a $(2,1)$-blocking set in $X/K$ and 
    for any $P\in\GrassmannDimSpace{0}{K}$, let $\blocking_{X/P}\subseteq\GrassmannDimSpace{0}{X/P}$ be the $(1,0)$-blocking sets in $X/P$ given by \autoref{cor:ballonConstruction}. 
    Then the disjoint union \[
    \blocking\leteq 
    \blocking(-1) \disjointunion 
    \blocking(0) \disjointunion 
    \blocking(2)\]
    is a $(2,1)$-blocking set in $X$,
    where the parts are defined by $\blocking(-1)\leteq \setBuilder{T\in\GrassmannDimSpace{1}{X}}{\generated{K,T}\in\blocking_{X/K}}$, by the disjoint union 
    $\blocking(0)\leteq \disjointUnion\setBuilder{\blocking_{X/P}\setminus\GrassmannDimSpace{0}{K/P}}{P\in \GrassmannDimSpace{0}{K}}$
    and $\blocking(2)\leteq \blocking_K$.
\end{theorem}
\begin{proof}[Proof of \autoref{lem:recursiveConsruction}]
    By \autoref{cor:ballonConstruction}, \autoref{lem:construction21ImprovedGeneral} is applicable and gives the statement.
\end{proof}

\begin{remark}\label{rem:MetschIsSpecialCase}
    Over finite fields, the $k=0$, $n=4$ case of this construction is essentially the same as in \cite{eis} where its minimality is also proved.
    The $k=1$, $n=5$ case is the same as \cite[Theorem 1.2]{Metsch}, and Metsch conjectures its minimality. 
    In accordance, the inequality of \autoref{recursion} generalises that of \autoref{thm:n<3s-2}. 
\end{remark}

\begin{remark}\label{optimalis}
    Note that the construction of \autoref{lem:recursiveConsruction} may not be optimal. 
    In \autoref{cor:ballonConstruction}, we constructed $\blocking_{X/P}$ so that $\GrassmannDimSpace{0}{K/P}\subseteq \blocking_{X/P}$ in order to minimise the disjoint union $\disjointUnion\setBuilder{\blocking_{X/P}\setminus\GrassmannDimSpace{0}{K/P}}{P\in \GrassmannDimSpace{0}{K}}$
    from \autoref{lem:construction21ImprovedGeneral}. 
    This ultimately imposed the condition $\dim(K)\leq \frac{\dim(X)-1}{2}$. 
    It may happen, if we require $\GrassmannDimSpace{0}{K/P}\cap \blocking_{X/P}$ to be large enough (but not necessarily the largest possible), then we may take a larger $K$ that may give altogether a smaller construction.

    Another way to potentially improve the construction is to carefully consider all possible choices of $t_i$, $s_i$ and $K$ from \autoref{partial}, c.f. \autoref{tab:tisi} for a natural choice, and \autoref{tab:tisi21} for a more complicated one giving rise to a smaller construction.
    Doing this analysis fully may be doable for small $(s,t)$, thereby finding the best possible construction of the form of \autoref{partial}.
    \end{remark}

The construction of \autoref{lem:recursiveConsruction} gives the following recursive upper bound for $f(n,q)$.
\begin{prop}[Recursive upper bound]\label{recursion}
    For any integer $n\geq -1$ and $-1\leq k\leq \frac{n-1}{2}$, we have     
    \begin{equation*}
      f(n,q)\leq q^{2k+2}f(n-k-1,q) + f(k,q) + \sum_{i=0}^k q^i \sum_{j=k}^{n-2} q^j
    \end{equation*}  
\end{prop}
\begin{proof}
    We can check the $n\leq 3$ cases directly using the known values of $f(n,q)$ from \autoref{defn:f} and \autoref{thm:f234}. 
    
    Assume now that $n\geq 4$. Note that in this case the assumption $k\leq \frac{n-1}{2}$ for the integer $k$ is equivalent to $k\leq \min\{\frac{n-1}{2},n-3\}$, so \autoref{lem:recursiveConsruction} is applicable. 
    Let $X=\PG(n,q)$,  $K\in\GrassmannDimSpace{k}{X}$ and 
    let $\blocking\leteq 
    \blocking(-1) \disjointunion 
    \blocking(0) \disjointunion 
    \blocking(2) \subseteq \GrassmannDimSpace{1}{X}$ be the $(2,1)$-blocking set given by 
    \autoref{lem:recursiveConsruction} where 
    we take $\blocking_{X/K}$ and $\blocking_K$ to be minimal. Note that $\blocking_{X/P}$ are minimal $(1,0)$-blocking sets by \autoref{lem:minimalBlockingForSmallST}.
    Then by \autoref{lem:constructionFibreSize} and the construction, we have the following.
    \begin{align*}
        \cardinality{\blocking(-1)}&=q^{2(k+1)}\cardinality{\blocking_{X/K}} = q^{2k+2}f(n-k-1,q)
        \\
        \cardinality{\blocking(0)} &=
        \sum_{\mathclap{P\in \GrassmannDimSpace{0}{K}}} \cardinality{\blocking_{X/P}\setminus\GrassmannDimSpace{0}{K/P}} = 
        \sum_{\mathclap{P\in \GrassmannDimSpace{0}{K}}} \grassmannDimSpace{0}{n-2,q}-\grassmannDimSpace{0}{k-1,q}  = 
        \grassmannDimSpace{0}{k,q} \sum_{j=k}^{n-2} q^j
        = \sum_{i=0}^k q^i \sum_{j=k}^{n-2} q^j
        \\
        \cardinality{\blocking(2)}&=\cardinality{\blocking_K}=f(k,q)
    \end{align*}
    Adding the equations in the cases above yields the statement.
\end{proof}

Unlike \autoref{rem:basinConstructionStepInvariant}, the construction in \autoref{lem:recursiveConsruction} can be improved by recursive applications. 
In \autoref{sec:recursiveUpperBOund}, we analyse the recursion of \autoref{recursion} and determine the optimal choice for $k$.

\subsection{Computing the recursive upper bound on \texorpdfstring{$f(n,q)$}{f(n,q)}}\label{sec:recursiveUpperBOund}
For a fixed prime power $q$, we start from the known exact values for $f(3,q)$ and $f(4,q)$ (see \autoref{thm:f234}), and use \autoref{recursion} to recursively find upper bounds for $f(n,q)$ for each $n\ge 5$. Iteratively for each $n\ge 5$, we would like to choose the value $0\le k\le \frac{n-1}{2}$ giving the best upper bound. We will now state and prove a result showing that the best value is the unique value $0\le k\le \frac{n-1}{2}$ such that $n-k$ is a power of $2$.

\begin{defn}
Let $n$ and $k$ be integers. Let us say that the pair $(n,k)$ is \emph{valid} if $-1\le k\le \frac{n-1}{2}$.
\end{defn}

\begin{defn}\label{one_step}
Let $n\ge 1$, and suppose that $\phi$ is a function defined on $[-1, n-1]$ such that $\phi(-1)=0$. For $-1\le k\le \frac{n-1}{2}$, denote by $u(\phi, n, q, k)$ the upper bound appearing in an application of \autoref{recursion}, that is:

$$u(\phi,n,q,k)=q^{2k+2}\phi(n-k-1)+\phi(k)+\sum_{i=0}^k q^i \sum_{j=k}^{n-2} q^j.$$
\end{defn}

Note that the choice of $k=-1$ corresponds to a trivial application of  \autoref{recursion} where the right-hand side is simply $\phi(n)$.

\begin{remark}
Note that to start the recursion when computing $f(n,q)$, we can use the trivial cases $f(-1,q)=f(0,q)=f(1,q)=0$, $f(2,q)=1$, and also the values $f(3,q)=q^2+1$ and $f(4,q)=q^4+2q^2+q+1$ as seen in \autoref{thm:f234}.
\end{remark}

\begin{defn}
Define the function $f^\opt(n,q)$ as follows. Let $f^\opt(n,q)=f(n,q)$ for all $-1\le n\le 4$, and then recursively define 
\[f^\opt(n,q)=\min_{0\le k\le \frac{n-1}{2}} u(f^\opt(\cdot,q),n,q,k)\]
for all $n\ge 5$.
\end{defn}

For a positive integer $n$, if we allow $0\le k\le \frac{n-1}{2}$ then the possible range of $n-k$ is $\left[\left\lceil \frac{n+1}{2}\right\rceil, n\right]$
, which contains a unique power of two.

\begin{defn}
For a positive integer $n$, let $k^{*}(n)$ be the unique value $0\le k\le \frac{n-1}{2}$ such that $n-k$ is a power of two. 
Namely, $k^*(n)=n-2^{\left\lfloor \log _2(n)\right\rfloor}$.
\end{defn}

\begin{prop}\label{powers_of_2}
For all $N\ge 5$ there exists $q_0(N)$ such that we have $f^\opt(N,q)=u(f^\opt(\cdot,q),N,q,k^*(N))$ for all prime powers $q\ge q_0(N)$.
\end{prop}

To prove this proposition, we consider an anchor sequence $f^*(n,q)$ which uses the values $k=k^*(n)$ and we prove two properties of this sequence, eventually showing that $f^\opt(n,q)=f^*(n,q)$ for all $-1\le n\le N$ and $q\ge q_0(N)$.

\begin{defn}\label{fstar_recursive}
For integers $n\ge -1$, define $f^*(n,q)$ as follows: let $f^*(n,q)=f(n,q)$ for $-1\le n\le 4$, and then recursively define $f^*(n,q)=u(f^*(\cdot,q),n,q,k^*(n))$ for all $n\ge 5$.
\end{defn}

By this definition, it is clear that $f^*(n,q)$ is a polynomial in $q$ for all $n$.

A polynomial in $q$ can be described by its \textit{coefficient sequence}, which starts with the leading coefficient, lists coefficients in order of decreasing power, and is considered to end with an infinite sequence of zeroes. For example, the coefficient sequence of the polynomial $x^3-5x+4$ is $(1,0,-5,4,0,0,0,...)$. Two polynomials (of possibly different degrees) can be lexicographically compared via their coefficient sequence: if $f(q),g(q)\in \mathbb{R}[q]$ have coefficient sequences $(a_1,a_2,a_3,...)$ and $(b_1,b_2,b_3,...)$, then we say $f<_\lex g$ if there exists $i\ge 1$ such that $a_i<b_i$, and $a_j=b_j$ for all $1\le j\le i-1$.

\begin{lemma}\label{starpoly_lemma}
The polynomials $f^*(n,q)$ satisfy the following:
\begin{enumerate}[label=(\alph*)]
\item For each $n\ge 2$, $\deg f^*(n,q)=2n-4$ and its leading coefficient is $1$.
\item For any $3\le n<n'$, we have $f^*(n,q)<_\lex f^*(n',q)$ and the first difference between the coefficient sequences is in the term of $n$-th highest order, where the two terms differ by exactly 1.
\end{enumerate}
\end{lemma}
\begin{proof}
For part (a), the statement clearly holds for $2\le n\le 4$. For $n\ge 5$, taking $k=k^*(n)$ we have

$$u(f^*(\cdot, q), n, q, k)=q^{2k+2}f^*(n-k-1,q)+f^*(k,q)+\sum_{i=0}^k q^i \sum_{j=k}^{n-2} q^j.$$

By induction, the first term has degree $(2k+2)+2(n-k-1)+4=2n-4$ with coefficient $1$, the second term has degree $2k-4<2n-4$ if $k\ge 2$ (or is $0$ if $0\le k\le 1$), and the third term has degree $n+k-2<2n-4$.

For part (b), it is sufficient to prove the result for $n'=n+1$ (where $n\ge 3$). We distinguish two cases.\smallskip

\noindent \textit{Case 1:} $n=2^d-1$ for some integer $d$.

In this case, $k^*(n+1)=0$ and so $f^*(n+1,q)=q^2f^*(n,q)+\sum_{j=0}^{n-1} q^j$. Therefore, $f^*(n+~1,q)-q^2f^*(n,q)$ has leading coefficient $q^{n-1}$, and since $\deg f^*(n+1,q)=2n-2$, this corresponds precisely to the $n$-th highest order term of the two polynomials.\smallskip

\noindent \textit{Case 2:} $2^d\le n\le 2^{d+1}-2$ for some integer $d$.

In this case, let $k=k^*(n)$; then $k^*(n+1)=k+1$. Denote $n_0=n-k-1$ which is one less than a power of 2. Then we have

$$f^*(n,q)=q^{2k+2}f^*(n_0,q)+f^*(k,q)+\sum_{i=0}^k q^i+\sum_{j=k}^{n-2} q^j$$

$$f^*(n+1,q)=q^{2k+4}f^*(n_0,q)+f^*(k+1,q)+\sum_{i=0}^{k+1} q^i\sum_{j=k+1}^{n-1} q^j$$

and we need to show that the leading term of $f^*(n+1,q)-q^2f^*(n,q)$ is $q^{n-1}$.

The difference is $f^*(k+1,q)-q^2f^*(k,q)+\sum_{j=k+1}^{n-1} q^j$. Here by the inductive hypothesis, $f^*(k+1,q)-q^2f^*(k,q)$ has leading term $q^{k-1}$ if $k\ge 3$ (or is constant if $0\le k\le 2$), so the top term $q^{n-1}$ dominates (it is easy to check that $k+1\le n-1$ for all $n\ge 2$, so this term does exist).

\end{proof}

In order to show \autoref{powers_of_2}, we will consider the `jumps' (i.e. values $k+1$) in applications of \autoref{recursion}. We will show in a crucial lemma that when applying the upper bound twice in a row (with a fixed sum of jumps), it is always advantageous to start with a jump as large as possible.

For this, we extend \autoref{one_step} to making two consecutive recursive steps using \autoref{recursion}: first with $k=k_1$ to bound $f(n)$, then with $k=k_2$ to bound $f(n-k_2-1)$.

\begin{defn}\label{two_step}
Let $n\ge 1$ and let $\phi$ be a function defined on $\{-1, 0, ..., n-1\}$ with $\phi(-1)=0$. If $k_1$ and $k_2$ are integers such that $(n,k_2)$ and $(n-k_2-1,k_1)$ are valid pairs, then let

$$u(\phi,n,q,k_1,k_2)=q^{2k_1+2}u(\phi,n-k_1-1,q,k_2)+\phi(k_1)+\sum_{i=0}^{k_1} q^i \sum_{j=k_1}^{n-2} q^j.$$
\end{defn}

\begin{lemma}\label{further_better}
Suppose that for integers $n\ge 1$ and $k_1,k_2,\ell_2\ge 0$ and $\ell_1\ge -1$, we have $k_1+k_2=\ell_1+\ell_2=d$ and $k_2<\ell_2$. Furthermore suppose that the pairs $(n+d+2,k_1)$, $(n+k_2+1,k_2)$, $(n+d+2, \ell_1)$, $(n+\ell_2+1, \ell_2)$ are valid. Then we have $$u(f^*(\cdot,q),n+d+2,q,k_1,k_2) > u(f^*(\cdot,q),n+d+2, q, \ell_1, \ell_2)$$ for $q$ large enough, and considering both sides as polynomials in $q$, the main term in the difference between the two sides is $q^{2k_1+k_2}$.
\end{lemma}

\begin{proof}
Write out both sides. The left hand side is: 

$$f^*(k_1,q)+q^{2k_1+2}u(f^*(\cdot,q), n+k_2+1, q, k_2) + \sum_{i=0}^{k_1} q^i \sum_{j=k_1}^{n+d} q^j$$

$$=f^*(k_1,q)+q^{2k_1+2}\left(f^*(k_2,q) + q^{2k_2+2}f^*(n,q) + \sum_{i=0}^{k_2} q^i\sum_{j=k_2}^{n+k_2-1} q^j\right)+\sum_{i=0}^{k_1} q^i\sum_{j=k_1}^{n+d} q^j$$

and similarly one can express the right hand side with $\ell_i$ in place of $k_i$, where the term $q^{2k_1+2k_2+4}f^*(n,q)=q^{2\ell_1+2\ell_2+4}f^*(n,q)$ appears on both sides. So the difference of the left hand and right hand sides is

\begin{align*}
D=&\underbrace{f^*(k_1,q)-f^*(\ell_1,q)}_{D_1}+\underbrace{q^{2k_1+2} f^*(k_2,q)-q^{2\ell_1+2}f^*(\ell_2,q)}_{D_2}\\&
+\underbrace{q^{2k_1+2}\cdot \sum_{i=0}^{k_2} q^i\cdot \sum_{j=k_2}^{n+k_2-1} q^j-q^{2\ell_1+2}\cdot \sum_{i=0}^{\ell_2} q^i\cdot \sum_{j=\ell_2}^{n+\ell_2-1} q^j+\sum_{i=0}^{k_1} q^i\cdot \sum_{j=k_1}^{n+d} q^j-\sum_{i=0}^{\ell_1} q^i\cdot \sum_{j=\ell_1}^{n+d} q^j}_{D_3}    
\end{align*}

Let us examine the degrees and leading terms of each subexpression.
\begin{itemize}
\item By part (a) of \autoref{starpoly_lemma}, $D_1$ has leading term $q^{2k_1-4}$, unless $k_1\le 1$, in which case $D_1\equiv 0$.
\item We examine the leading term of $D_2$, using both parts of the Lemma.
\begin{itemize}
\item If $k_2\ge 3$ then the degrees of both $q^{2k_1+2}f^{*}(k_2,q)$ and $q^{2\ell_1+2}f^{*}(\ell_2,q)$ are $2d-2$, and the first difference is in the $k_2$-th highest term, giving that the leading term of $D_2$ is $-q^{2d-k_2-1}$.
\item If $k_2=2$ then the degrees of both terms are $2d-2$. In this case, if $\ell_2\ge 4$, then the third-highest term of $f^{*}(\ell_2,q)$ is one greater than that of $f^{*}(3,q)=q^2+1$, so it is $2$, whereas the second-highest term of $f^{*}(n,q)$ is equal (zero) for all $n\ge 2$. So the leading term of $D_2$ will be $-2q^{2d-4}$ in this case. If $\ell_2=3$ then the leading term is $-q^{2d-4}$.
\item If $k_2\le 1$ and $\ell_2\ge 2$ then $f^{*}(k_2,q)=0$ whereas $f^{*}(\ell_2,q)$ has leading term $q^{2\ell_2-4}$, therefore $D_2$ has leading term $-q^{2\ell_1+2}q^{2\ell_2-4}=-q^{2d-2}$.
\item Finally, if $k_2=0$ and $\ell_2=1$ then $D_2\equiv 0$.
\end{itemize}
\item To calculate $D_3$, we substitute the sum of each geometric sequence involved.
\begin{align*}
D_3=\frac{1}{(q-1)^2}&(q^{2k_1+2}(q^{k_2+1}-1)(q^{n+k_2}-q^{k_2})-q^{2\ell_1+2}(q^{\ell_2+1}-1)(q^{n+\ell_2}-q^{\ell_2})\\&
+(q^{k_1+1}-1)(q^{n+d+1}-q^{k_1})-(q^{\ell_1+1}-1)(q^{n+d+1}-q^{\ell_1}))
\end{align*}

After expanding, we can eliminate many of the terms, including all terms with $n$ in the exponent, using the identities $k_1+k_2=\ell_1+\ell_2=d$. What remains is

$$\frac{1}{(q-1)^2}(q^{d+k_1+2}-q^{d+\ell_1+2}-q^{2k_1+1}+q^{2\ell_1+1}+q^{k_1}-q^{\ell_1})$$

where in the brackets, the sole term of highest degree is $q^{d+k_1+2}$, meaning that the leading term of $D_3$ is $q^{d+k_1}$.
\end{itemize}

So now let us compute the overall leading term of $D$. We have seen that $D_3$ always has a nonvanishing leading term $q^{d+k_1}$, whereas $D_1$ (if it does not vanish) has a smaller-order leading term $q^{2k_1-4}$ (note that $2k_1-4<d+k_1$ is equivalent to $k_1-4<d$ which holds since $k_2\ge 0$), and the degree of the leading term of $D_2$ (if it does not vanish) is at most $2d-k_2-1$ in all cases. Here we have $2d-k_2-1<d+k_1 $, equivalently,     $d-1<k_1+k_2=d$ which  holds by our assumption. Overall, the leading term of $D$ is $q^{d+k_1}=q^{2k_1+k_2}$ in all cases. This also means that the left hand side in the statement of the Lemma is greater than the right hand side for $q$ large enough.
\end{proof}

\begin{remark}\label{onestep_better}
The case $\ell_1=-1$ of \autoref{further_better} means that whenever it is possible to combine two jumps ($k_1+1$ and $k_2+1$) into one jump which has their combined size ($\ell+1=d+2$), then it is always advantageous to do so for $q$ large enough.
\end{remark}

\begin{proof}[Proof of \autoref{powers_of_2}]
By induction we prove that for every $N\ge -1$ there exists $q_0(N)$ such that $f^{opt}(n,q)=f^{*}(n,q)$ holds for all $n\le N$ and $q\ge q_0(N)$. If we can do this then the Proposition clearly follows.

For $-1\le N\le 4$ this holds for $q_0(N)=2$ as both sequences are equal to $f(n,q)$ for all $q$ by definition. So now take $N\ge 5$.

Taking $q'=q_0(N-1)$, for each $q\ge q'$ it holds that $f^{opt}(n,q)=f^{*}(n,q)$ for all $n\le N-1$. So we just need to show that for large enough $q$ we also have $f^{opt}(N,q)=f^{*}(N,q)$, so

$$\min_{0\le k\le \frac{N-1}{2}} u(f^{opt}(\cdot,q),N,q,k)=u(f^{*}(\cdot,q),N,q,k^*(N)).$$

Equivalently, since $f^{opt}=f^*$ for values smaller than $N$, we need

\begin{equation}\label{eq:star_min_needed}
\min_{0\le k\le \frac{N-1}{2}} u(f^*(\cdot,q),N,q,k)=u(f^{*}(\cdot,q),N,q,k^*(N)).
\end{equation}

Recall that $k^*(N)$ is the unique value $0\le k\le \frac{N-1}{2}$ such that $N-k$ is a power of two; say $N-k=2^{t}$ where $\ell\ge 2$ is an integer (as $N\ge 5$). To show \eqref{eq:star_min_needed}, we will show that for each $0\le k'\le \frac{N-1}{2}$ with $k'\ne k$, we have $u(f^*(\cdot,q),N,q,k')\ge u(f^*(\cdot,q),N,q,k^*(N))$ for $q$ large enough.

\textbf{Case 1:} $0\le k'<k^*(N)$

In this case, note that $k^*(N-k'-1)=k^*(N)-k'-1$, since $N-k'-1-(k^*(N)-k'-1)=N-k^*(N)=2^{t}$, and $0\le k^*(N)-k'-1\le \frac{(N-k'-1)-1}{2}$, where the second inequality is proven by $k^*(N)-k'-1\le \frac{N-1}{2}-k'-1\le \frac{N-k'-2}{2}$, which is true since $k'\ge 0$.

We use the case of \autoref{further_better} mentioned in \autoref{onestep_better}: take $n=N-k^*(N)-1$, $k_1=k'$, $k_2=k^*(N)-k'-1$, $\ell_1=-1$ and $\ell_2=k^*(N)$. Note that in particular, the pair $(N-k_1-1, k_1)$ is valid by the previous observation. Then for $q$ large enough we get

\begin{equation}\label{eq:powersof2_case1}
u(f^*(\cdot, q), N, q, k^*(N))<u(f^*(\cdot, q), N, q, k', k^*(N)-k'-1).
\end{equation}

Since $k^*(N-k'-1)=k^*(N)-k'-1$, the right hand side of \eqref{eq:powersof2_case1} is equal to $u(f^*(\cdot,q), N, q, k')$, as required.

\textbf{Case 2:} $k^*(N)<k'\le \frac{N-1}{2}$

In this case, we show that $k^*(N-k'-1)=N-k'-2^{t-1}-1$. Firstly we have $$N'-k'-1-(N'-k'-2^{t-1}-1)=2^{t-1}.$$ Secondly $$N-k'-2^{t-1}-1\ge 0 \mathrm{ \ \ is \ equivalent \ to \ \ } N-k'>2^{t-1},$$ and the latter holds since $N-k'\ge N-\frac{N-1}{2}=\frac{N+1}{2}>2^{t-1}$, as $N\ge N-k^*(N)=2^t$. Thirdly, $$N-k'-2^{t-1}-1\le \frac{(N-k'-1)-1}{2}=\frac{N-k'}{2}-1, \mathrm{ \ \ that \ is, \ \ } N-k'\le 2^t$$ holds, since $N-k'<N-k^*(N)=2^t$.

Also it is evident that $k^*(2^t-1)=2^{t-1}-1$.

Now we use \autoref{further_better} with $n=2^{t-1}-1$, $k_1=k'$, $k_2=N-k'-2^{t-1}-1$, $\ell_1=k^*(N)$ and $\ell_2=2^{t-1}-1$. By the previous observations, $(N-k_1-1,k_2)$ and $(N-\ell_1-1,\ell_2)$ are indeed valid pairs. We get that for $q$ large enough,

$$u(f^*(\cdot,q), N, q, k', N-k'-2^{t-1}-1)>u(f^*(\cdot,q), N, q, k^*(N), 2^{t-1}-1).$$

Now since $k^*(N-k'-1)=N-k'-2^{t-1}-1$ and $k^*(N-k^*(N)-1)=k^*(2^{t-1})=2^{t-1}-1$, this means that

$$u(f^*(\cdot,q), N, q, k')>u(f^*(\cdot,q), N, q, k^*(N)),$$

giving the required statement.
\end{proof}
The following table shows the coefficient sequences of the polynomials $f^*(n,q)$ for $2\le n\le 9$, obtained by computer using \autoref{fstar_recursive}. According to the definition of coefficient sequences that we have given, each sequence is considered to end with infinitely many zeroes.

\begin{table}[htb!]
\centering
\begin{tabular}{c||l}
$n$ & Coefficients\\ \hline\hline
2 & $1$\\ \hline
3 & $1,0,1$\\ \hline
4 & $1,0,2,1,1$\\ \hline
5 & $1,0,2,2,2,1,0$\\ \hline
6 & $1,0,2,2,3,2,1,0,1$\\ \hline
7 & $1,0,2,2,3,3,2,1,1,0,1$\\ \hline
8 & $1,0,2,2,3,3,3,2,2,1,2,1,1$\\ \hline
9 & $1,0,2,2,3,3,3,3,3,2,3,2,2,1,0$\\ 
\end{tabular}
\caption{Coefficient sequences of the polynomials $f^*(n,q)$.}
\label{tab:coefficients}
\end{table}

Using \autoref{starpoly_lemma}, we get that the coefficient sequence starts with $(1,0,2,2,3,3,3,3)$  for all $n\ge 9$, and it is lexicographically smaller than this for $n<9$, so as a conclusion, 
we  proved the statement of \autoref{main2}, one of the main statements of the paper.

\begin{remark}
We could also choose a larger value $n_0$ (instead of $9$) and read off the first $n_0-1$ coefficients of $f^*(n_0,q)$ with further computations, using it to give a more refined upper bound for $f(n,q)$.
\end{remark}
\section{Lower bounds}\label{sect:lowerb}

In this section, we compare the $q$-analog Schönheim bound with the lower bound gained from the application of the method of standard equations. While the latter bound is slightly weaker, it pinpoints better the rooms for improvement and describes how to improve further the best known lower bound.

We will often refer to the number of incidences between a set of subspaces and certain Grassmannians. These can be seen as the degrees in the corresponding incidence graphs.
\begin{defn}[Degree]
    For any $\blocking\subseteq \GrassmannDimSpace{t}{X}$ and $Y\in\GrassmannDimSpace{y}{X}$, define the \emph{$\blocking$-degree of $Y$} to be 
    \[\deg_\blocking(Y)\leteq \begin{cases}
        \cardinality{\setBuilder{T\in\blocking}{Y\subseteq T}}, &y\leq t\\
        \cardinality{\setBuilder{T\in\blocking}{Y\supseteq T}}, &y\geq t\\
    \end{cases}.\]
\end{defn}

   For a projective space $X$ over $\field{q}$ and $t,s$ with $t\leq s\leq \dim(X)=n$, 
    an $(s,t)$-blocking set $\blocking=\blocking_t^s$ is a subset of $\GrassmannDimSpace{t}{X}$ such that $\cardinality{\deg_{\blocking}(S)}\geq 1$ for every $S\in\GrassmannDimSpace{s}{X}$.
  Recall that such $\blocking$ is also called a $q$-Turán design and its minimum cardinality is $\cT_q(n+1, s+1, t+1)$.

\begin{lemma}[Trivial lower bound]
    For every $t\leq s\leq n$ and $q$, we have 
    \[\TuranQNTS{q}{n}{s}{t} \geq \frac{\grassmannDimSpace{s}{n,q}}{\grassmannDimSpace{s-t-1}{n-t-1,q}}.
\]
    In particular, when $t=1$, $s=2$, we have $f(n,q)=\TuranQNTS{q}{n}{2}{1}  \geq \frac{(q^{n+1}-1)(q^n-1) }{(q^3-1) (q^2-1)}$. \\
    For $n\geq 4$, this gives 
    \[f(n,q) \geq \coefficients{1}{0}{1}{1}.
     \]
\end{lemma}
\begin{proof}

Let $\blocking\subseteq \GrassmannDimSpace{1}{X}$ be an $(s,t)$-blocking set in $X=\PG(n,q)$. 
Define $A\leteq \setBuilder{(T,S)\in\blocking\times \GrassmannDimSpace{s}{X}}{T\subset S}$.
On one hand, 
\[\cardinality{A}
=\sum_{T\in\blocking} \grassmannDimSpace{s-t-1}{X/T}
=\cardinality{\blocking}\cdot\grassmannDimSpace{s-t-1}{n-t-1,q}\]
as $X/T\isom \PG(n-t-1,q)$.
On the other hand, 
\[\cardinality{A}
=\sum_{S\in\GrassmannDimSpace{s}{X}} \deg_\blocking(S) 
\geq \sum_{S\in\GrassmannDimSpace{s}{X}} 1
= \grassmannDimSpace{s}{n,q}\]
using the definition of $\blocking$. 
Comparing these two lines gives the statement.
\end{proof}

The statement below is based on the tool called the standard equation method, which improves the previous result with a term of order $q^{2n-6}$.

\begin{lemma}
For arbitrary $n\geq 2$ and $q$, we have 
    \[f(n,q) 
    \geq 
    \frac{\grassmannDimSpace{0}{n}\grassmannDimSpace{1}{2} \grassmannDimSpace{2}{n}}{\grassmannDimSpace{0}{1}^2\grassmannDimSpace{1}{n}}
    = \frac{(q^{n+1}-1) (q^{n-1}-1) }{(q^2-1)^2}
    .\]
    
    In particular, for $n\geq 5$, this gives 
    \[f(n,q)\geq \coefficients{1}{0}{2}{0}\]
\end{lemma}

\begin{proof}
    Let $\blocking\subseteq \GrassmannDimSpace{1}{X}$ be a $(2,1)$-blocking set in $X\leteq \PG(n,q)$, 
    and write $\deg\leteq \deg_\blocking$ for simplicity. 
    Define a function $F$ by $F(x)=x(x-1)$. 
    Consider the following configurations and count their cardinality in different ways, where the sums and sets  range over $P\in X$, 
    $T,T_1,T_2\in\blocking$ and 
    $S\in\GrassmannDimSpace{2}{X}$. 
 
    \begin{gather}
        \sum_{P} 1 = \grassmannDimSpace{0}{n,q} \label{eq:countP}
        \\
        \sum_{P} \deg(P) 
        =\cardinality{\setBuilder{(P,T)}{P\in T}} 
        = \sum_{T} \grassmannDimSpace{0}{T} 
        = \cardinality{\blocking} \cdot \grassmannDimSpace{0}{1,q} 
        \label{eq:countPl}
        \\
        \sum_{P} F(\deg(P)) 
        = \cardinality{\setBuilder{(P,T_1,T_2,S)}{P\in T_i\subset S,T_1\neq T_2}}
        = \sum_{S} F(\deg(S))  
        \label{eq:countPllPi}
        \\
        \cardinality{\blocking}\cdot \grassmannDimSpace{0}{n-2,q} 
        =\sum_{T} \grassmannDimSpace{0}{X/T} 
        = \cardinality{\setBuilder{(T,S)}{T\subset S}} 
        = \sum_{S} \deg(S) 
        \label{eq:countlPi}
        \\
        \grassmannDimSpace{2}{n,q} 
        = \sum_{S} 1 \label{eq:countPi}
    \end{gather}

    To give estimates, on one hand,  we may apply Jensen's inequality for the convex function $F$ giving 
    \begin{equation} \label{eq:Jensen}
        F\left(\frac{\sum_{P} \deg(P)}{\sum_{P} 1}\right) \leq \frac{\sum_{P} F(\deg(P))}{\sum_{P} 1}.
    \end{equation}
    On the other hand, note that $1\leq \deg(S)\leq \grassmannDimSpace{1}{2,q}$, so 
    \begin{equation}\label{eq:fEstimate}
        F(\deg(S))=\deg(S)\cdot (\deg(S)-1)\leq  \grassmannDimSpace{1}{2,q} \cdot (\deg(S)-1).
    \end{equation}

    Putting the previous equations and inequalities together 
    gives
    \begin{align*}
        F\left(\frac{\cardinality{\blocking} \cdot \grassmannDimSpace{0}{1,q}} {\grassmannDimSpace{0}{n,q}} \right) 
        &\overunderset{\eqref{eq:countPl}}{\eqref{eq:countP}}{=} F\left(\frac{\sum_{P} \deg(P)}{\sum_{P} 1}\right)
        \overset{\eqref{eq:Jensen}}{\leq} \frac{\sum_{P} F(\deg(P))}{\sum_{P} 1}
        \\
        &\overunderset{\eqref{eq:countPllPi}}{\eqref{eq:countP}}{=} \frac{\sum_{S} F(\deg(S))}{\grassmannDimSpace{0}{n,q}}
        \overset{\eqref{eq:fEstimate}}{\leq} \frac{\sum_{S} \grassmannDimSpace{1}{2,q} \cdot (\deg(S)-1)}{\grassmannDimSpace{0}{n,q}}
        \\
        &\overunderset{\eqref{eq:countlPi}}{\eqref{eq:countPi}}{=} \frac{\grassmannDimSpace{1}{2}}{\grassmannDimSpace{0}{n,q}} \left(\cardinality{\blocking}\cdot \grassmannDimSpace{0}{n-2,q} - \grassmannDimSpace{2}{n,q} \right).
    \end{align*}
    Rearranging the two sides gives
    \[
    \cardinality{\blocking}^2
    - 
    \frac{\grassmannDimSpace{0}{n,q}\cdot\left(\grassmannDimSpace{0}{1,q}+ \grassmannDimSpace{1}{2,q} \grassmannDimSpace{0}{n-2,q}\right)}{\grassmannDimSpace{0}{1,q}^2}
     \cdot
    \cardinality{\blocking}
    +
    \frac{\grassmannDimSpace{0}{n,q}\grassmannDimSpace{1}{2,q} \grassmannDimSpace{2}{n,q}}{\grassmannDimSpace{0}{1,q}^2}
    \leq 0,
    \]
    hence $\cardinality{\blocking}$ lies between the roots of this quadratic function. 
    Note that $\cardinality{\blocking}=\grassmannDimSpace{1}{n,q}$ is one root, as for $\blocking=\GrassmannDimSpace{1}{X}$, all inequalities above are sharp. 
    Using Vieta's formula for the product of the roots, we obtain the other (smaller) one as in the statement.
\end{proof}

\begin{remark}
Note that when \eqref{eq:Jensen} and \eqref{eq:fEstimate} are both sharp, we have a strong condition on the degrees, as $\deg(P)$ is independent of $P\in X$ and 
    $\deg(S)\in\{1,\grassmannDimSpace{1}{2,q}\}$ for every $S\in\GrassmannDimSpace{2}{X}$.  In fact, in order to improve the lower bound above, a natural way is to derive \begin{itemize}
        \item results on the deviation of the degrees of points $\deg(P)$,
        \item results on  the deviation of the degrees of subspaces $\deg(S)$.
    \end{itemize}
    These in turn can be applied to improve the bounds gained from the calculation above.
\end{remark}

The following simple statement will be used to verify the density increment, \autoref{increment}.

\begin{lemma}\label{lem:restrictionIsBlocking}
    For $t\leq s\leq k\leq n$, 
    let $\blocking\subseteq \GrassmannDimSpace{t}{X}$ be an $(s,t)$-blocking set of $X$
    and let $K\in\GrassmannDimSpace{k}{X}$ be arbitrary. 
    Then $\blocking|_K\leteq \setBuilder{T\in \blocking}{T\subseteq K}$ is an $(s,t)$-blocking set of $K$. 
\end{lemma}
\begin{proof}
    Pick $S\in\GrassmannDimSpace{s}{K}$. 
    Then we also have $S\in\GrassmannDimSpace{s}{X}$, so by the definition of $\blocking$ being a blocking set in $X$, there is $T\in\blocking$ such that $T\subseteq S$. 
    But then $T\subseteq S\subseteq K$ implies $T\in \blocking|_K$.
\end{proof}

\begin{defn}[Density of line sets]
    Define the \emph{density function} by \[\density(n,q)\leteq \frac{f(n,q)}{\grassmannDimSpace{1}{n,q}}=\frac{f(n,q)}{\qbinom{n+1}{2}{q}}.\]
\end{defn}

\begin{lemma}[Density increment]\label{increment}
    The function $n\mapsto \density(n,q)$ is increasing for every $q$.
\end{lemma}

Note that this claim is essentially equivalent to a special case of the $q$-analog Schönheim bound.
\begin{prop}[$q$-analog Schönheim bound, Etzion, Vardy, \cite{Etz2}]
    $\cC_q(n,k,r)\geq \frac{q^n-1}{q^k-1}\cC_q(n-1,k-1,r-1).$
\end{prop}

For the sake of completeness, we give a short proof. Note that the same argument gives back the increasing property of the density of arbitrary $(s,t)$-blocking sets, with respect to the set of all $t$-subspaces.

\begin{proof}[Proof of Lemma \ref{increment}]
    Let $2\leq k\leq n$.
    Let $\blocking\subseteq \GrassmannDimSpace{1}{X}$ be a $(2,1)$-blocking set of $X\leteq\PG(n,q)$. 
    Double counting $A\leteq \setBuilder{(T,K)\in \blocking\times \GrassmannDimSpace{k}{X}}{T\subseteq K}$ gives 
    \begin{equation}\label{eq:densityDoubleCount}
        \cardinality{\blocking}\cdot \grassmannDimSpace{k-2}{n-2,q}
        = \sum_{T\in \blocking} \grassmannDimSpace{k-2}{X/T}
        = \cardinality{A}
        = \sum_{\mathclap{K\in \GrassmannDimSpace{k}{X}}} \deg_\blocking(K)
    \end{equation}
    because $X/T\isom \PG(n-2,q)$.
    First, apply \eqref{eq:densityDoubleCount} to a minimal blocking set $\blocking$, i.e. $\cardinality{\blocking}=f(n,q)$. Now $\deg_\blocking(K) = \deg_{\blocking|_K}(K) \geq f(k,q)$ by \autoref{lem:restrictionIsBlocking}, so 
    \[f(n,q) \cdot \grassmannDimSpace{k-2}{n-2,q}
    \geq \grassmannDimSpace{k}{X}\cdot f(k,q).\]
    Second, apply \eqref{eq:densityDoubleCount} to $\blocking=\GrassmannDimSpace{1}{X}$. 
    Now $\deg_\blocking(K)=\grassmannDimSpace{1}{K}=\grassmannDimSpace{1}{k,q}$, so  
    \[\grassmannDimSpace{1}{n,q} \cdot \grassmannDimSpace{k-2}{n-2,q}
    = \grassmannDimSpace{k}{X} \cdot \grassmannDimSpace{1}{k,q}.\]
    Dividing the last two observations by each other gives the statement.
\end{proof}

\begin{problem}\label{mainQ}
    Determine  $\lim_{n\to\infty}\density(n,q)$.
\end{problem}

Observe that the monotonicity and boundedness of the function show that the limit exists. Note also that \autoref{main2} yields $$\frac{1}{q^2+q}\leq \density(4,q)\leq \lim_{n\to\infty}\density(n,q)\leq   \frac{1}{q^2+q-2q^{-1}-2q^{-2}+4q^{-3}}.$$
\section{Concluding remarks and open problems}

 Referring to Etzion's table \cite{Etz1} on the best known results concerning the covering problem, in \autoref{tab:1} we show the improvement we made for the case $(s,t)=(2,1)$. The lower bound comes from a recursive application of \autoref{also}, whereas the previously known upper bounds come from Eisfeld and Metsch's results (see \autoref{rem:MetschIsSpecialCase}) for $n\in \{4,5\}$, and from the $n-1\mapsto n$ recursive construction (also coming from \cite[Theorem 7]{Etz1}) for $n\ge 6$. Note that in view of \autoref{optimalis}, our presented construction might be applicable by a different parameter set to gain even more on the upper bound. We believe that the exact values are actually closer to the upper bounds, as it was claimed for $n=5$ by Metsch \cite{Metsch}.

\begin{table}[ht!]
\centering
\begin{minipage}{0.45\textwidth}
\centering
\begin{tabular}{c||c|c|c}
 $q=2$ & lower & known upper & new upper \\ \hline\hline
 $n=4$ & 27 & 27 & 27 \\ \hline
 $n=5$ & 114 & 122 & 122 \\ \hline
 $n=6$ & 468 & 519 & 517 \\ \hline
 $n=7$ & 1895 & 2139 & 2125 \\ \hline
 $n=8$ & 7625 & 8683 & 8627 \\ \hline
 $n=9$ & 30590 & 34987 & 34762
\end{tabular}
\end{minipage}
\hspace{0.05\textwidth}
\begin{minipage}{0.45\textwidth}
\centering
\begin{tabular}{c||c|c|c}
 $q=3$ & lower & known upper & new upper \\ \hline\hline
 $n=4$ & 103 & 103 & 103 \\ \hline
 $n=5$ & 938 & 966 & 966 \\ \hline
 $n=6$ & 8474 & 8815 & 8812 \\ \hline
 $n=7$ & 76360 & 79699 & 79660 \\ \hline
 $n=8$ & 687520 & 718384 & 718033 \\ \hline
 $n=9$ & 6188519 & 6468736 & 6465576
\end{tabular}
\end{minipage}

\caption{The previously known bounds on  $f(n,q)=\cT_q(n+1,3,2)=\cC_q(n+1,n-1,n-2)$  (c.f. \cite{Etz1}) and our new results. 
In general, for $n\ge 6$, the leading term in the series expansion of the difference between the known and new upper bound (as a function of $q$) is $1\cdot q^{2n-11}$.}
\label{tab:1}
\end{table}

Throughout the paper, our main focus was to give bounds on $(2,1)$-blocking sets. However, the main idea of the recursive refinement presents an approach that can improve general cases as well, albeit it would be more involved. This is the reason we did not want to elaborate on it.
Many statements of the paper hold in the general case which could be the starting point of a more general argument, c.f. \autoref{increment}, \autoref{partial}, especially for $t=s-1$.

The paper discussed a new bound on $q$-covering and $q$-Turán designs, which raised increasing interest
 due to their connections with constant-dimension codes. In particular, a
$q$-Steiner design is an optimal constant-dimension code.
 Metsch \cite{Metsch3} conjectured that no such designs $S_q[n,k,r]$ exist, provided that 
$k > r > 1$. However, this turned out to be false, as Braun, Etzion, Östergård, Vardy, and Wassermann showed the existence of 
 $S_2[13, 3, 2]$ $2$-Steiner designs \cite{Braun}.

Very recently,  Keevash,  Sah and Sawhney proved the existence of subspace designs \cite{Keevash} with any given parameters, provided that the dimension of the underlying space is sufficiently large in terms of the other parameters of the design and satisfies the obvious necessary divisibility conditions,  which settled the corresponding open problem from the 1970s.
In particular, for $s> t\geq 1$ fixed integers and $q$ fixed prime power, they showed that if $n$ attains a certain threshold, then once $\qbinom{k-i}{r-i}{q} \bigm\vert  \qbinom{n-i}{r-i}{q}$ holds for all $0\leq i\leq r-1$, a subspace design $S_q[n,k,r]$ exists. However, this cannot be applied in our case, since from the dual of the $(2,1)$-blocking sets in $\PG(n,q)$, the corresponding parameter set of $\cC_q[n+1,n-1,n-2]$, namely $k=n-1$, $r=n-2$ will not be independent of $n$.

We mention the connection point between maximum rank distance codes (MRD codes) and covering designs. This has been discussed by Pavese \cite{Pavese} who used lifted MRD codes to improve bounds on $\mathcal{C}_q[2n,4, 3]$ and on $\mathcal{C}_q[3n+~8, 4,2]$.

At last, we reiterate  our main open problem: to determine the limit of the density $\lim_{n\to\infty}\density(n,q)$.

\end{document}